\documentclass[12pt,letterpaper]{amsart}
\usepackage{amssymb, amsmath, amsthm, esint}
\usepackage{mathrsfs}
\usepackage{bbm,dsfont}
\usepackage{hyperref}
\usepackage{mathtools}
\usepackage{fullpage}
\usepackage{color}
\usepackage{comment}
\usepackage[dvipsnames]{xcolor}
\usepackage[shortcuts]{extdash}

\title
{Pointwise convergence of bilinear polynomial averages over the primes}

\author[B. Krause]{Ben Krause}
\address{BK: Department of Mathematics, University of Bristol, BS8 1QU, UK}
\email{ben.krause@bristol.ac.uk}

\author[H. Mousavi]{Hamed Mousavi}
\address{HM: Department of Mathematics, University of Bristol, BS8 1QU, UK}
\email{gj23799@bristol.ac.uk}

\author[T. Tao]{Terence Tao}
\address{TT: Department of Mathematics, University of California, Los Angeles, CA 90095-1555,USA}
\email{tao@math.ucla.edu}

\author[J. Ter\"{a}v\"{a}inen]{Joni Ter\"{a}v\"{a}inen}
\address{Department of Pure Mathematics and Mathematical Statistics, University of Cambridge, Cambridge CB3 0WB, UK}
\email{joni.p.teravainen@gmail.com}


\newcommand{\R}{\mathbb{R}}
\newcommand{\N}{\mathbb{N}}
\newcommand{\C}{\mathbb{C}}
\newcommand{\Z}{\mathbb{Z}}
\newcommand{\T}{\mathrm{T}}
\newcommand{\Q}{\mathbb{Q}}
\newcommand{\A}{\mathbb{A}}
\newcommand{\V}{\mathbf{V}}
\newcommand{\E}{\mathbb{E}}
\newcommand{\n}{\mathrm{n}}
\newcommand{\D}{\mathbb{D}}
\newcommand{\I}{\mathbb{I}}
\newcommand{\Height}{\mathrm{h}}
\newcommand{\Avg}{\mathrm{A}}
\newcommand{\B}{\mathrm{B}}
\newcommand{\M}{\mathcal{M}}
\newcommand{\F}{\mathcal{F}}
\newcommand{\eps}{\varepsilon}
\newcommand{\Log}{\operatorname{Log}}
\newcommand{\Cramer}{\operatorname{Cram\acute{e}r}}
\newcommand{\HB}{\operatorname{HB}}
\newcommand{\ind}[1]{\mathds{1}_{{#1}}}
\renewcommand{\mod}{{\ \mathrm{mod}\ }}
\let\oldpmod\pmod
\renewcommand{\pmod}[1]{\hspace{-0.12cm}\oldpmod {#1}}

\theoremstyle{plain}

\newtheorem{theorem}{Theorem}[section]
\newtheorem{proposition}[theorem]{Proposition}
\newtheorem{lemma}[theorem]{Lemma}
\newtheorem{corollary}[theorem]{Corollary}

\theoremstyle{remark}
\newtheorem*{remark}{Remark}

\numberwithin{equation}{section}
\begin{document}

\begin{abstract}
We show that on a $\sigma$-finite measure preserving system $X = (X,\nu, T)$, the non-conventional ergodic averages
$$ \E_{n \in [N]} \Lambda(n) f(T^n x) g(T^{P(n)} x)$$
converge pointwise almost everywhere for $f \in L^{p_1}(X)$, $g \in L^{p_2}(X)$, and $1/p_1 + 1/p_2 \leq 1$, where $P$ is a polynomial with integer coefficients of degree at least $2$.  This had previously been established with the von Mangoldt weight $\Lambda$ replaced by the constant weight $1$ by the first and third authors with Mirek, and by the M\"obius weight $\mu$ by the fourth author.  The proof is based on combining tools from both of these papers, together with several Gowers norm and polynomial averaging operator estimates on approximants to the von Mangoldt function of ``Cram\'er'' and ``Heath-Brown'' type.
\end{abstract}

\maketitle

\section{Introduction}

Throughout this paper,  $P \in \Z[\n]$ denotes a polynomial with integer coefficients of some degree $d \geq 2$ in one indeterminate $\n$; a typical case to keep in mind is the quadratic polynomial $P(\n) = \n^2$.

Define a \emph{measure-preserving system} to be a triple $X = (X,\nu,T)$, where $X = (X,\nu)$ is a $\sigma$-finite measure space, and $T \colon X \to X$ is an invertible bimeasurable map which is measure-preserving in the sense that $\nu(T^{-1}(E)) = \nu(E)$ for all measurable $E$.  It is common in the literature to restrict to finite measure systems, and to normalize $\nu(X)=1$; but our results will not require any hypothesis of finite measure. Given functions $f,g \colon X \to \C$, a scale $N \geq 1$, and a weight function $w \colon \N \to \C$, we can then define the non-conventional averaging operator
$$ \Avg_{N, w; X}(f,g)(x) \coloneqq \E_{n \in [N]} w(n) f(T^n x) g(T^{P(n)} x)$$
for any $x \in X$ (see Section \ref{notation-sec} for our averaging notation).

\subsection{Unweighted ergodic averages}

In the unweighted case $w=1$, the following ergodic theorem was recently proven by two of the authors with Mirek.

\begin{theorem}[Unweighted ergodic theorem]\label{unweight-thm}\cite[Theorem 1.17]{KMT0} Let $(X,\nu,T)$ be a measure-preserving system, and let $f \in L^{p_1}(X)$, $g \in L^{p_2}(X)$ for some $1 < p_1,p_2 < \infty$ with $\frac{1}{p_1} + \frac{1}{p_2} = \frac{1}{p} \leq 1$.
\begin{itemize}
\item[(i)] (Mean ergodic theorem)  The averages $\Avg_{N,1;X}(f,g)$ converge in $L^p(X)$ norm.
\item[(ii)] (Pointwise ergodic theorem)  The averages $\Avg_{N,1;X}(f,g)$ converge pointwise almost everywhere.
\item[(iii)] (Maximal ergodic theorem)  One has
$$ \| (\Avg_{N,1;X}(f,g))_{N \in \Z^+} \|_{L^p(X; \ell^\infty)} \lesssim_{p_1,p_2,P} \|f\|_{L^{p_1}(X)} \|g\|_{L^{p_2}(X)}$$
(see Section \ref{asymptotic-sec} for our asymptotic notation conventions).
\item[(iv)] (Variational ergodic theorem)  If $r>2$ and $\lambda>1$, one has
$$ \| (\Avg_{N,1;X}(f,g))_{N \in \D} \|_{L^p(X; \V^r)} \lesssim_{p_1,p_2,r,P,\lambda} \|f\|_{L^{p_1}(X)} \|g\|_{L^{p_2}(X)}$$
whenever $\D \subset [1,+\infty)$ is finite and $\lambda$-lacunary (see Section \ref{var-def} for the definition of $\lambda$-lacunarity and the variational norm $\V^r$).
\end{itemize}
\end{theorem}

We very briefly review the main ingredients of the proof of Theorem \ref{unweight-thm}.  Part (iv) is the main estimate, which easily implies the other three claims.  By some standard sparsification and transference arguments, as well as dyadic decompositions, it sufficed to prove the variant estimate
$$ \| (\tilde \Avg_{N,1}(f,g))_{N \in \D} \|_{\ell^p(\Z; \V^r)} \lesssim_{p_1,p_2,r,P,\lambda} \|f\|_{\ell^{p_1}(\Z)} \|g\|_{\ell^{p_2}(\Z)}$$
where
\begin{equation}\label{avg-upper-def}
    \tilde \Avg_{N, w}(f,g)(x) \coloneqq \E_{n \in [N]} w(n) f(x+n) g(x+P(n)) \ind{n>N/2}
\end{equation}
is the ``upper half'' of $\Avg_{N,w;X}$ when $X$ is the integers $\Z$ with the usual shift $T \colon n \mapsto n+1$ and counting measure $\nu$.

A crucial observation was that the averages $\tilde \Avg_{N,1}$ are ``complexity zero'' in the sense that they are small when the Fourier transform of $f$ or $g$ vanish on ``major arcs''.  Indeed, in \cite[Theorem 5.12]{KMT0} the single-scale minor arc estimate
\begin{equation}\label{single}
     \| \tilde \Avg_{N,1}(f,g) \|_{\ell^1(\Z)} \lesssim_{C_1} ( 2^{-cl} + \langle \Log N \rangle^{-cC_1} ) \|f\|_{\ell^2(\Z)} \|g\|_{\ell^2(\Z)}
\end{equation}
was proven for $N \geq 1$, $l \in \N$, and $f,g \in \ell^2(\Z)$ with either the Fourier transform $\F_\Z f$ of $f$ vanishing on the major arc set $\M_{\leq l, \leq -\Log N + l}$ or the Fourier transform $\F_\Z g$ of $g$ vanishing on the major arc set $\M_{\leq l, \leq -d\Log N + dl}$; we refer the reader to Section \ref{notation-sec} for the definition of the various terms and symbols introduced here.  This minor arc estimate was proven by combining Peluse--Prendiville estimates \cite{PP} with a discrete $\ell^p$ improving inequality from \cite{koselacey}, together with a Hahn--Banach argument.

Using \eqref{single}, one could now focus attention to major arcs.  After some routine manipulations involving Ionescu--Wainger multiplier theory \cite{IW}, the task reduced to controlling the $\ell^p(\Z; \V^r)$ norm of tuples of the form
\begin{equation}\label{afg-1}
     (\tilde \Avg_{N,1}(F_N, G_N))_{N \in \I}
\end{equation}
where $\I$ is a certain $\lambda$-lacunary set (bounded from below by certain bounds, but not from above), and $F_N, G_N$ are various frequency localizations of $f,g$ respectively to major arcs (see \cite[Theorem 5.30]{KMT0} for a precise statement).  By estimation of the bilinear symbol of the averaging operator $\tilde \Avg_{N,1}$, one could approximate this tuple by another tuple
\begin{equation}\label{afg-2} 
    (\B^{l_1, l_2, m_{\hat \Z}}_{(\varphi_N \otimes \tilde \varphi_N) \tilde m_{N,\R}}(F,G) )_{N \in \I}
\end{equation}
where $F,G$ are again some Fourier localizations of $f,g$ to major arcs, and $\B^{l_1, l_2, m_{\hat \Z}}_{(\varphi_N \otimes \tilde \varphi_N) \tilde m_{N,\R}}$ is a certain bilinear Fourier multiplier adapted to major arcs; see \cite[Proposition 7.13]{KMT0} for a precise statement.  At this stage it became necessary to split the set $\I$ of spatial averaging scales into the small scales $\I_{\leq}$ and large scales $\I_{>}$.  For the small scales, one could reduce matters to controlling another tuple
$$ (\B^{l_1,l_2, m_{\hat \Z}}_{m_*}(\T^{l_1}_{\varphi_{N,t,j_1}} F, \T^{l_2}_{\tilde{\varphi}_{N,t,j_2}} G ))_{N \in \I_{\leq}}$$
for another bilinear Fourier multiplier $B^{l_1,l_2, m_{\hat \Z}}_{m_*}$ and Fourier multipliers $T^{l_1}_{\varphi_N,t,j_1}$, $T^{l_2}_{\tilde{\varphi}_N,t,j_2}$, while for the large scales one instead considered tuples of the form
$$ (\B_{1 \otimes m_{\hat \Z}}( \T_{\varphi_{N,t,j_1} \otimes 1} F_\A, \T_{\tilde{\varphi}_{N,t,j_2} \otimes 1} G_\A ))_{N \in \I_{>}}$$
where $F_\A, G_\A$ were now defined on the ring $\A_\Z = \R \times \hat \Z$ of adelic integers rather than on the integers $\Z$.  See \cite[Theorem 7.28]{KMT0} for a precise statement of the estimates required on these tuples.

In the small-scale case, it was possible to apply a general two-parameter Radamacher--Menshov inequality \cite[Corollary 8.2]{KMT0} followed by some shifted Calder\'on--Zygmund theory \cite[Theorem B.1]{KMT0} to reduce matters to obtaining a good $\ell^{p_1}(\Z) \times \ell^{p_2}(\Z) \to \ell^p(\Z)$ estimate for the bilinear multiplier $\B^{l_1,l_2, m_{\hat \Z}}_{m_*}$ (see \cite[Lemma 8.6]{KMT0}), which was ultimately proven with the assistance of the minor arc estimate \eqref{single} and the approximation result in \cite[Proposition 7.13]{KMT0}.

In the large scale case, some interpolation and factorization arguments, together with a version of \eqref{single} on the profinite integers $\hat \Z$, reduced matters to establishing $L^2(\Z_p) \times L^2(\Z_p) \to L^q(\Z_p)$ bounds on the $p$-adic averaging operator
\begin{equation}\label{avg-zp}
 \Avg_{\Z_p}(f,g)(x) \coloneqq \E_{n \in \Z_p} f(x+n) g(x+P(n))
\end{equation}
for all primes $p$ and some $q>2$, with the operator norm required to be bounded by $1$ for $p$ large enough; see \cite[(10.3), (10.4)]{KMT0} for a precise statement.  The boundedness ultimately came from some distributional analysis of the level sets of $P$ on the $p$-adics (see \cite[Corollary C.2]{KMT0}); getting the bound of $1$ for large $p$ required some additional refined analysis in which one again uses (a $p$-adic version of) the minor arc estimate \eqref{single}.

\subsection{M\"obius-weighted ergodic averages}

More recently, another one of us \cite{joni} considered the non-conventional averaging operators $\Avg_{N,\mu;X}$ weighted by the M\"obius function $\mu$ instead of $1$.  Perhaps counter-intuitively, the convergence of ergodic averages weighted by $\mu$ is actually \emph{better} than that of the unweighted case, especially in light of the recent progress on quantitative Gowers uniformity of the M\"obius function \cite{gt-mobius, TT, leng-quadratic, leng-equidistribution, lss}.  For instance, as a special case of \cite[Theorem 1.2]{joni}, the following result was shown.

\begin{theorem}[M\"obius-weighted ergodic theorem]\label{mobius-ergodic}  Let  $X$ have finite measure, $f \in L^{p_1}(X)$, $g \in L^{p_2}(X)$ with $\frac{1}{p_1} + \frac{1}{p_2} < 1$, and let $A>0$.  Then
\begin{equation}\label{nnfg}
     \lim_{N \to\infty} (\log^A N) \Avg_{N,\mu;X}(f,g) = 0
\end{equation}
pointwise almost everywhere.
\end{theorem}

The ingredients used to prove Theorem \ref{mobius-ergodic} are somewhat different from those used to prove Theorem \ref{unweight-thm}; a key input was \cite[Theorem 4.1]{joni}, which in our context establishes the bound
\begin{equation}\label{unif-est}  
    | \E_{x \in [-CN^d, CN^d]} \Avg_{N,\theta;\Z}(f,g)(x) h(x) | \lesssim_{C,P} (N^{-1} + \|\theta\|_{u^{d+1}[N]})^{1/K} \end{equation}
for all $1$-bounded $f,g,h, \theta$ and some $1 \leq K \lesssim_d 1$, where the ``little'' Gowers uniformity norm $\|\theta\|_{u^{d+1}[N]}$ is defined as
\begin{equation}\label{little-gowers} \|\theta\|_{u^{d+1}[N]} \coloneqq \sup_{\deg Q \leq d} |\E_{n \in [N]} \theta(n) e(-Q(n))| \end{equation}
where $Q$ ranges over all polynomials of degree at most $d$ with real coefficients, and $e(x) \coloneqq e^{2\pi i x}$.  The results of \cite{gt-mobius} show that
$\|\mu\|_{u^{d+1}[N]}$ decays faster than any power of $\log N$, and the claim then follows by standard sparsification and transference arguments.

\subsection{Prime-weighted ergodic averages}

In this paper we combine the methods of \cite{joni} and \cite{KMT0}, together with some additional arguments, to obtain a non-conventional ergodic theorem in which the weight is selected to be the von Mangoldt function $\Lambda$, defined by
\begin{align*}
\Lambda(n)=\begin{cases}\log p,\quad n \textnormal{ is a power of a prime } p,\\
0,\quad \quad \,\,\,\textnormal{otherwise}.\end{cases}    
\end{align*}
 More specifically, we show the following. 

\begin{theorem}[Main theorem]\label{main-thm}  Let $(X,\nu,T)$ be a measure-preserving system, and let $f \in L^{p_1}(X)$, $g \in L^{p_2}(X)$ for some $1 < p_1,p_2 < \infty$ with $\frac{1}{p_1} + \frac{1}{p_2} \leq 1$.  Then the averages $\Avg_{N,\Lambda;X}(f,g)$ converge pointwise almost everywhere.  In fact, one has the variational estimate
\begin{equation}\label{variational}
    \| (\Avg_{N,\Lambda;X}(f,g))_{N \in \D} \|_{L^p(X; \V^r)} \lesssim_{p_1,p_2,p,r,P,\lambda} \|f\|_{L^{p_1}(X)} \|g\|_{L^{p_2}(X)}
\end{equation}
whenever $\lambda > 1$, $p \geq 1$ and $r>2$ with $\frac{1}{p_1} + \frac{1}{p_2} = \frac{1}{p}$, and $\D\subset [1,+\infty)$ is finite and $\lambda$-lacunary.
\end{theorem}

The range of $r$ here is optimal, as will be mentioned in Subsection~\ref{sec:sharpness}. It is possible to extend the the range of $(p_1,p_2)$ slightly beyond duality, see the discussion in Subsection~\ref{sec:duality}. 

Using the fact that $\log n=\log N+O(\log M)$ for $n\in [N/M,N]$ and the prime number theorem, we have the following immediate corollary to Theorem~\ref{main-thm}.

\begin{corollary}\label{main-cor}
Let the assumptions be as in Theorem~\ref{main-thm}. Then the prime-weighted averages
$$ \frac{1}{N/\log N} \sum_{p \leq N} f(T^p x) g(T^{P(p)} x)$$
converge pointwise almost everywhere. 
\end{corollary}

Previously, the pointwise convergence of ergodic averages over the primes was known only in the case of a single polynomial iterate. This case was established by Bourgain~\cite{bourgain-arithmetic} and Wierdl~\cite{wierdl} 
 for linear polynomials (with the latter work allowing $L^q$ functions for any $q>1$), and the case of an arbitrary single polynomial iterate was handled by Nair~\cite{nair1},~\cite{nair2}. We also mention that the problem of pointwise convergence of ergodic averages with more than one iterate was discussed by Frantzikinakis in~\cite[Problem 12]{frantzikinakis-survey}; the specific problem there about two linear iterates however remains open. 

 Let us also mention that the \emph{norm convergence} of non-conventional ergodic averages is now known for any number of polynomial iterates, thanks to the works of Frantzikinakis--Host--Kra~\cite{fhk} and Wooley--Ziegler~\cite{wooley-ziegler}.

\subsection{Methods of proof}

From a high-level perspective, Theorem \ref{main-thm} is proven by combining the methods used in \cite{KMT0} to prove Theorem \ref{unweight-thm} with the methods used in \cite{joni} to prove Theorem \ref{mobius-ergodic}.  However, several technical difficulties make the analysis delicate in places, as we shall now discuss.

The first issue arises when trying to approximate various frequency-localized averages (analogous to \eqref{afg-1}, but with the weight $1$ replaced by $\Lambda$) by certain bilinear model operators (analogous to \eqref{afg-2}, but with the symbol $m_{\hat \Z}$ replaced by a variant $m_{\hat \Z^\times}$).  It is important for the arguments in \cite{KMT0} that the error in this approximation gains a polynomial factor $N^{-c}$ in $N$, or at least a quasipolynomial factor $\exp(-\log^c N)$.  Using the von Mangoldt function as a weight, this is possible in the absence of Siegel zeroes (and in particular assuming the generalized Riemann hypothesis); however, the presence of a Siegel zero near a given scale $N$ requires one to add a scale-dependent correction term to the bilinear symbol $m_{\hat \Z}$ in order to obtain a satisfactory approximation at small scales.  While this correction term is ultimately manageable because of the Landau--Page theorem, it significantly complicates the analysis, in that one cannot simply repeat arguments from \cite{KMT0} verbatim.  See Section \ref{Remarks-sec} for further discussion.

In order to avoid this issue, we adapt some ideas from \cite{TT} and swap the von Mangoldt weight $\Lambda$ early in the argument with an approximant $\Lambda_N$ that is not sensitive to Siegel zeroes.  The arguments used in \cite{joni} to establish Theorem \ref{mobius-ergodic} allow one to do so provided that one has good control of the little Gowers uniformity norm in the sense that
$$    \|\Lambda - \Lambda_N\|_{u^{d+1}[N]} \lesssim \langle \Log N \rangle^{-A}
$$
for some large $A$.  One available choice of approximant is the \emph{Cram\'er(--Granville) approximant}
$$ \Lambda_{\Cramer, w}(n) \coloneqq \frac{W}{\varphi(W)} \ind{(n,W)=1}$$
for a suitable parameter $w$ and $W=\prod_{p\leq w}p$ (we end up selecting $w \coloneqq \exp(\Log^{1/C_0} N)$ for some large constant $C_0$); the required bounds follow for instance from the results in \cite{mstt} (which even extend to shorter intervals).  A useful fact, first observed in \cite{TT} and refined further here, is that these approximants are stable in Gowers uniformity norms with respect to the $w$ parameter; see Lemma \ref{cramer-stable} for a precise statement.

After using the arguments from \cite{joni} to replace $\Lambda$ by $\Lambda_N$, most of the arguments of \cite{KMT0} proceed with only minor changes; in particular, the analogue of the approximation of \eqref{afg-1} by \eqref{afg-2} is fairly routine, thanks in large part to the fundamental lemma of sieve theory; see the proof of Proposition \ref{mod-approx} in Section \ref{verifying-sec}.  We remark that Siegel zeroes play no role whatsoever in establishing this proposition, in contrast to what would have occurred if we retained the original weight $\Lambda$ instead of $\Lambda_N$. However, three components of the argument of Theorem \ref{main-thm} still require some additional care. The first is a \emph{polynomial improving estimate}
\begin{align*}
    \Big( \sum_{x\in \mathbb{Z}} |\E_{n\in [N]} (\Lambda(n)+\Lambda_N(n))f(x+P(n)) |^2 \Big)^{1/2}
   \lesssim N^{d(1/2-1/p)} \|f\|_{\ell^p(\Z)}
\end{align*}
for $p\in (2-c_P,2]$, with $c_P>0$ small (see Lemma~\ref{lp-improv}). This is eventually reduced to the analogous unweighted improving estimate using some properties of the  Cram\'{e}r approximant, in particular Corollary \ref{mean}.

The second component is the $p$-adic estimates, in which the averaging operator \eqref{avg-zp} ends up being replaced by the variant
$$
    \Avg_{\Z_p^\times}(f,g)(x) \coloneqq \E_{n \in \Z_p^\times} f(x+n) g(x+P(n)).
$$
It is necessary to bound the $L^2(\Z_p) \times L^2(\Z_p) \to L^q(\Z_p)$ norm of this operator by exactly the constant $1$ when $q>2$ is close to $2$ and $p$ is large; losing a multiplicative factor such as $1+O(1/p)$ would not be acceptable as one needs to multiply these constants over all primes $p$.  Fortunately, the effect of restricting to the invertible elements $\Z_p^\times$ of $\Z_p$ is not too severe, and the arguments from \cite{KMT0} can be adapted with only a modest amount of effort to avoid any losses of $O(1/p)$ in the constants.

The most delicate step is to adapt the single-scale estimate \eqref{single} to the weighted setting.  As the Peluse--Prendiville theory is somewhat complicated, our approach is to use the approximation theory from \cite{joni} to try to replace the approximant $\Lambda_N$ with an approximant closer to the constant weight $1$.  With the theory of the Cram\'er approximant from \cite{TT}, it is not too difficult to replace $\Lambda_N$ by a Cram\'er approximant $\Lambda_{\Cramer, w}$ for a smaller parameter $w$, with error terms polynomial in $w$.  However, a technical problem then arises: this approximant is not a pure ``Type I'' sum of the form $\sum_{d\mid n} \lambda_d$ for certain well-behaved weights $\lambda_d$, preventing one from removing the weight entirely.  To resolve this, we appeal to the theory from \cite{joni} once more to replace the Cram\'er approximant $\Lambda_{\Cramer, w}$ with a more Fourier-analytic approximant, which we call the \emph{Heath-Brown approximant} (as it was introduced by him in \cite{HB85}). This approximant is defined by
$$     \Lambda_{\HB,Q}(n) \coloneqq \sum_{q < Q} \frac{\mu(q)}{\varphi(q)} c_q(n) $$
where $Q$ is a parameter of similar size to $w$, and $c_q$ is a Ramanujan sum; roughly speaking, this approximant is the main term in the Fourier restriction of the von Mangoldt function to major arcs.  By using the analysis of the little Gowers uniformity norms of Type I sums from \cite{mat1}, we are able to show that $\Lambda_{\Cramer, w}$ is close in these norms to 
$\Lambda_{\HB,w}$, and then by the theory from \cite{joni} (and a dyadic decomposition), one can replace the former by the latter, at least for the purposes of proving an ``$\ell^\infty$'' Peluse--Prendiville inverse theorem for weighted averages.  As in \cite{KMT0}, it is also necessary to obtain a more delicate ``$\ell^2$'' inverse theorem, which requires a weighted version of the $\ell^p$ improving inequality from \cite{koselacey}, but this can be achieved by a variant of the arguments just presented.

\begin{remark}
The proof of Theorem~\ref{main-thm} quickly yields a version of Peluse's inverse theorem~\cite[Theorem 3.3]{Peluse} with prime weights. This was not needed for proving Theorem~\ref{main-thm} (what we did need was in essence a version with the weight function $\Lambda_N$; see Proposition~\ref{lip-wt}), but we believe such a result may be of independent interest, so we record it as Theorem~\ref{peluse-thm}. Some combinatorial applications of this result will be investigated in a future work.
\end{remark}

\begin{remark}
We expect the methods of this paper to be applicable to be applicable also to pointwise convergence of bilinear polynomial ergodic averages weighted by some other weights of arithmetic interest. The exact requirements for the weight are not so easy to axiomatize, but we need the weight to satisfy analogues of~\eqref{lambda-approx}--\eqref{imp}, as well as a suitable ``local-to-global'' factorization over the primes to be able to pass to the adeles. In particular, we expect the methods to be applicable to ergodic averages weighted by the divisor function $\tau$, but we will not pursue this problem here.
\end{remark}

\subsection{Acknowledgments} We thank the referee for careful reading of the paper. 

BK is supported by an EPSRC New Investigators grant and an ERC Starting grant.

TT is supported by NSF grant DMS-2347850.

JT is supported by European Union's Horizon
Europe research and innovation programme under Marie Sk\l{}odowska-Curie grant agreement No. 101058904 and ERC grant agreement No. 101162746, and Academy of Finland grant No. 362303. 

\section{Notation}\label{notation-sec}

\subsection{General notation}

Our notation largely follows \cite{KMT0}, though somewhat abridged, as some of the notation in \cite{KMT0} is only used to establish results or arguments that we are treating here as ``black boxes''.

We use $\Z_+ \coloneqq \{1,2,\dots\}$ to denote the positive integers and $\N \coloneqq \{0,1,2,\dots\}$ to denote the natural numbers.  

We use $\ind{E}$ to denote the indicator function of a set $E$. Similarly, if $S$ is a statement, we use $\ind{S}$ to denote its indicator, equal to $1$ if $S$ is true and $0$ if $S$ is false. Thus for instance $\ind{E}(x) = \ind{x \in E}$. We use $|E|$ to denote the cardinality of a set $E$, and adopt for $f\colon E\to \mathbb{C}$ the averaging notation
$$ \E_{n \in E} f(n) \coloneqq \frac{1}{|E|} \sum_{n \in E} f(n)$$
if $E$ is finite and non-empty.  We similarly define $L^p$ norms
$$ \|f\|_{L^p(E)} \coloneqq \left( \sum_{n \in E} |f(n)|^p\right)^{1/p}$$
for $0 < p < \infty$, with the usual convention that $\|f\|_{L^\infty(E)}$ is the (essential) supremum of $f$ on $E$.  One can extend these averaging conventions to other measurable spaces $E$ of positive finite measure (such as a $p$-adic group $\Z_p$ equipped with Haar probability measure), if $f$ (or $|f|^p$) is absolutely integrable, in the obvious fashion. When $X$ is equipped with counting measure, we will write  $\ell^p(X)$ or just $\ell^{p}$ in place of  $L^p(X)$.

Throughout, $p'$ denotes the dual exponent of $p\in [1,\infty]$, so $1/p+1/p'=1$.

If $f \colon X \to \C$, $g \colon Y \to \C$ are functions, we use $f \otimes g \colon X \times Y \to \C$ to denote the tensor product
$$ (f \otimes g) (x,y) \coloneqq f(x) g(y).$$

\subsection{Magnitudes and asymptotic notation}\label{asymptotic-sec}

We use the Japanese bracket notation
$$ \langle x \rangle \coloneqq (1 + |x|^2)^{1/2}$$
for any real or complex $x$.  We use $\lfloor x \rfloor$ to denote the greatest integer less than or equal to $x$. 
For any $N \geq 1$ we define the \emph{logarithmic scale} $\Log N$ of $N$ by the formula
\begin{equation}\label{log-scale}
 \Log N \coloneqq \lfloor \log N / \log 2 \rfloor
 \end{equation}
thus $\Log N$ is the unique natural number such that $2^{\Log N} \leq N < 2^{\Log N+1}$.

For any two quantities $A, B$ we will write
$A \lesssim B$, $B \gtrsim A$, or $A = O(B)$ to denote the bound
$|A| \leq CB$ for some absolute constant $C$.  If we need the implied constant $C$ to depend on additional parameters we will denote this by subscripts, thus for instance $A \lesssim_\rho B$ denotes the bound $|A| \leq C_\rho B$ for some $C_\rho$ depending on $\rho$.  We write $A \sim B$ for $A \lesssim B \lesssim A$.  To abbreviate the notation we will sometimes explicitly permit the implied constant to depend on certain fixed parameters (such as the polynomial $P$) when the issue of uniformity with respect to such parameters is not of relevance.  Due to our reliance in some places\footnote{Specifically, Siegel's theorem is used in \cite{mstt}, and we will use results from that paper to establish \eqref{lambda-approx}.} on tools based on Siegel's theorem, several of the implied constants in our arguments will be ineffective, but we will not track the effectivity of constants explicitly in this paper.

\subsection{Algebraic notation}
If $R$ is a commutative ring, we use $R^\times$ to denote the multiplicatively invertible elements of $R$.

\subsection{Number theoretic notation}

For any $N > 0$, $[N]$ denotes the discrete interval $[N] \coloneqq \{ n \in \Z_+\colon n \leq N \}$.  
If $q_1,q_2 \in \Z_+$, we write $q_1\mid q_2$ if $q_1$ divides $q_2$. If $a,q \in \Z_+$, we let $(a,q)$ denote the greatest common divisor of $a$ and $q$, and $[a,q]$ the least common multiple.

All sums and products over the symbol $p$ will be understood to be over primes; other sums will be understood to be over positive integers unless otherwise specified.

In addition to the von Mangoldt function $\Lambda(n)$ and M\"obius function $\mu(n)$ already introduced, we will also use the divisor function $\tau(n) \coloneqq \sum_{d\mid n} 1$ and the Euler totient function $\varphi(n) \coloneqq |(\Z/n\Z)^\times|$.

\subsection{Fourier analytic notation}\label{fourier-sec}

We write $e(\theta) \coloneqq e^{2\pi i \theta}$ for any real $\theta$, and also $\|\theta\|_{\R/\Z}$ for the distance from $\theta$ to the nearest integer.

For a prime $p$, we let $\Z_p$ be the ring of $p$-adic integers, defined as the inverse limit of the cyclic groups $\Z/p^j\Z$ for $j \in \N$; this is a compact abelian group equipped with a Haar probability measure.  Similarly, let $\hat \Z$ be the ring of profinite integers, defined as the inverse limit of the cyclic groups $\Z/Q\Z$ for all positive integers $Q$; this is again a compact abelian group with a Haar probability measure, being the direct product of the $\Z_p$.  We use $\E_{\Z_p}$ or $\E_{\hat \Z}$ to denote averaging with respect to these compact abelian groups.  Finally, we let $\A_\Z \coloneqq \R \times \hat \Z$ denote the ring of adelic integers, which is a locally compact abelian group.

We define some Fourier transforms on various locally compact abelian groups:
\begin{itemize}
\item[(i)] Given a summable function $f \colon \Z \to \C$, the Fourier transform $\F_\Z f \colon \R/\Z \to \C$ is defined by the formula
$$ \F_\Z f(\theta) \coloneqq \sum_{n \in \Z} f(n) e(-n\theta).$$
\item[(ii)] Given a Schwartz function $f \colon \R \to \C$, the Fourier transform $\F_\R f \colon \R \to \C$ is defined by the formula
$$ \F_\R f(\xi) \coloneqq \int_\R f(x) e(-x\xi) \ dx.$$ 
\item[(iii)] Given a function $f \colon \hat \Z \to \C$ which is \emph{Schwartz--Bruhat} in the sense that it factors through a function $f_Q \colon \Z/Q\Z \to \C$ on a cyclic group, we define the Fourier transform $\F_{\hat \Z} f \colon \Q/\Z \to \C$ by the formula
$$ \F_{\hat \Z} f\left(\frac{a}{Q} \mod 1\right) \coloneqq \E_{n \in \Z/Q\Z} f_Q(n) e(-an/Q)$$
for any integer $a$.
\item[(iv)] Given a function $f \colon \A_\Z \to \C$ which is Schwartz--Bruhat in the sense that it factors through a function $f_Q \colon \R \times \Z/Q\Z$ which is Schwartz in the first variable, we define the Fourier transform $\F_\A f \colon \R \times \Q/\Z \to \C$ by the formula
$$ \F_{\hat \A} f\left(\xi, \frac{a}{Q} \mod 1\right) \coloneqq \E_{n \in \Z/Q\Z} \int_\R f_Q(x,n) e(-x\xi - an/Q)\ dx$$for
 integer $a$ and $\xi \in \R$, and $\F_{\hat \A}$ vanishing otherwise.
\end{itemize}
We refer the reader to \cite[\S 4]{KMT0} for a further discussion of the Fourier transform on such locally compact abelian groups as $\Z$, $\R$, $\Z_p$, $\hat \Z$, $\Z/Q\Z$ or $\A_\Z$, and the various intertwining relationships between these transforms.

Given a Schwartz symbol $m \colon \R/\Z \to \C$, we define the Fourier multiplier $\T_m$ on $\ell^2(\Z)$ by the formula
$$ \T_m f(x) \coloneqq \int_{\R/\Z} m(\xi) \F_\Z f(\xi) e(x\xi) \ d\xi,$$
and similarly given a bilinear Schwartz symbol $m \colon \R/\Z \times \R/\Z \to \C$, define the bilinear Fourier multiplier $\B_m$ by the formula
$$ \B_m(f,g)(x) \coloneqq \int_{\R/\Z} \int_{\R/\Z} m(\xi,\eta) \F_\Z f(\xi) \F_\Z g(\eta) e(x(\xi+\eta)) \ d\xi d\eta.$$
Linear and bilinear multipliers are defined similarly for the other locally compact abelian groups defined here, and obey a certain operator calculus; again, we refer the reader to \cite[\S 4]{KMT0} for details, as we shall largely use facts and arguments about these operators from \cite{KMT0} as ``black boxes''.

We will need the Ionescu--Wainger Fourier multipliers on major arcs.  Again, we shall mostly be using these tools as ``black boxes'', so their definition and properties are not of critical importance in this paper; but for sake of completeness we recall the main definitions from \cite{KMT0}.  Given a small parameter $\rho$, it is possible to assign a \emph{Ionescu--Wainger height} $\Height(\alpha)=\Height_{\rho}(\alpha) \in 2^\N$ for each $\alpha \in \Q/\Z$; see \cite[Appendix A]{KMT0}.  Using this height, we define the Ionescu--Wainger arithmetic frequency sets
$$ (\Q/\Z)_{\leq l} \coloneqq \Height^{-1}([2^l]) = \{\alpha \in \Q/\Z \colon \Height(\alpha) \leq 2^l\}$$
and the Ionescu--Wainger major arcs
\begin{equation}\label{major-arcs}
{\mathcal M}_{\leq l, \leq k} \coloneqq \{ \xi + \alpha\colon \xi \in \R, |\xi| \leq 2^k, \alpha \in (\Q/\Z)_{\leq l} \},
\end{equation}
thus ${\mathcal M}_{\leq l, \leq k}$ is the union of arcs $[\alpha-2^k, \alpha+2^k]$ for $\alpha \in (\Q/\Z)_{\leq l}$; we will be focused on the regime where $k$ is sufficiently small that these arcs are disjoint, which happens whenever $k \leq -C_\rho 2^{\rho l}$.
We also use the variants
$$ (\Q/\Z)_l \coloneqq (\Q/\Z)_{\leq l} \backslash (\Q/\Z)_{\leq l-1} =  \Height^{-1}(2^l) = \{ \alpha \in \Q/\Z \colon \Height(\alpha)=2^l\},$$
and
$$  {\mathcal M}_{l, \leq k} \coloneqq  {\mathcal M}_{\leq l, \leq k} \backslash  {\mathcal M}_{\leq l-1, \leq k}$$
with the convention that $(\Q/\Z)_{\leq -1}$ and ${\mathcal M}_{\leq -1,k}$ are empty.

The  \emph{Ionescu--Wainger Fourier projection operator} $\Pi_{\leq l, \leq k}$ for any $(l,k) \in \N \times \Z$ is defined by the formula
$$ \Pi_{\leq l, \leq k} f(x) = \sum_{\alpha \in (\Q/\Z)_{\leq l}} \int_\R \eta(\theta/2^k) \F_\Z f(\alpha+\theta) e(-x(\alpha+\theta))\ d\theta$$
where $\eta$ is a smooth even function supported on $[-1,1]$ that equals $1$ on $[-1/2,1/2]$. We then define
$$ \Pi_{l, \leq k} \coloneqq \Pi_{\leq l,\leq k} - \Pi_{\leq l-1,\leq k}.$$
We refer the reader to \cite[\S 5, Appendix A]{KMT0} for the key properties of these projections, which can be viewed as analogues of Littlewood--Paley projection operators for major arcs.

\subsection{Variational norms}\label{var-def}

A sequence $1 \leq N_1 < \dots < N_k$ of positive reals is said to be \emph{$\lambda$-lacunary} for some $\lambda \geq 1$ if
$$ N_{j+1}/N_j > \lambda$$
for all $1 \leq j < k$.

For any finite dimensional normed vector space $(B,\|\cdot\|_B)$ and any sequence
 $(\mathfrak a_t)_{t\in\I}$ of elements of $B$ indexed by a totally
 ordered set $\I$, and any exponent $1 \leq r < \infty$, the
 $r$-variation seminorm is defined by the formula
 \begin{equation}\label{var-seminorm}
  \| (\mathfrak a_t)_{t \in \I} \|_{V^r(\I; B)} \coloneqq
 \sup_{J\in\Z_+} \sup_{\substack{t_{0} \leq \dotsb \leq t_{J}\\ t_{j}\in\I}}
\Big(\sum_{j=0}^{J-1}  \|\mathfrak a(t_{j+1})-\mathfrak a(t_{j})\|_B^{r} \Big)^{1/r},
 \end{equation}
where the  supremum is taken over all finite increasing sequences in $\I$, and is set by convention to equal zero if $\I$ is empty.

The $r$-variation norm for $1 \leq r < \infty$ is defined by
\begin{equation}\label{vardef}
  \| (\mathfrak a_t)_{t \in \I} \|_{\V^r(\I; B)} 
\coloneqq \sup_{t\in\I}\|\mathfrak a_t\|_B+
\| (\mathfrak a_t)_{t \in \I} \|_{V^r(\I;B)}.
\end{equation}
This clearly defines a norm on the space of functions from $\I$ to $B$.
If $B=\C$, then we will abbreviate $V^r(\I;X)$ to $V^r(\I)$ or $V^r$, and $\V^r(\I;X)$ to $ \V^r(\I)$ or $\V^r$. 

\subsection{Gowers norms}

In addition to the little Gowers uniformity norm $u^{d+1}[N]$ defined in \eqref{little-gowers}, we will also need the full Gowers norm $U^{d+1}[N]$ defined for functions $f \colon \Z \to \C$ as
$$ \|f\|_{U^{d+1}[N]} \coloneqq \| f \ind{[N]} \|_{U^{d+1}(\Z)} / \|\ind{[N]}\|_{U^{d+1}(\Z)}$$
where the $U^{d+1}(\Z)$ norm is defined for finitely supported functions by the formula
$$ \|f\|_{U^{d+1}(\Z)}^{2^{d+1}} \coloneqq \sum_{x,h_1,\dots,h_{d+1} \in \Z} \prod_{\omega \in \{0,1\}^{d+1}} {\mathcal C}^{\omega_1+\dots+\omega_{d+1}} f(x + \sum_{j=1}^{d+1} \omega_j h_j)$$
where $\omega = (\omega_1,\dots,\omega_{d+1})$, and ${\mathcal C}$ denotes the complex conjugation operator.  It is well known that
\begin{equation}\label{udn}
 \| f\|_{u^{d+1}[N]} \lesssim_d \|f\|_{U^{d+1}[N]};
\end{equation}
see, e.g., \cite[(2.2)]{gt-u3}.

Similar uniformity norms $u^{d+1}(I)$, $U^{d+1}(I)$ can then be defined for other intervals $I \subset \R$ than $[N]$ in the obvious fashion.

\section{High-level proof of theorem}

We now describe the high-level proof of Theorem \ref{main-thm}, reducing it to two key statements (Theorem~\ref{ssmae} and Proposition~\ref{mod-approx})  that we will prove in Section~\ref{verifying-sec}.  The arguments here will closely follow those of \cite{KMT0}, and some familiarity with the arguments in that paper would be highly recommended in order to follow the text in this section.

In the next section we shall introduce an approximant $\Lambda_N \colon \N \to \R$ to $\Lambda$ (depending on a parameter $C_0$) which enjoys the bound
\begin{equation}\label{lambda-approx}
    \|\Lambda - \Lambda_N\|_{u^{d+1}[N]} \lesssim_{A,C_0} \langle \Log N \rangle^{-A}
\end{equation}
for any $A>0$, as well as the pointwise bound
\begin{equation}\label{tau-n}
     \Lambda_N(n) \lesssim_{C_0} \langle \Log N \rangle^{O(1)},
\end{equation}
the $L^1$ bound
\begin{equation}\label{L1}
   \E_{n \in [N]} |\Lambda_N(n)| \lesssim_{C_0} 1,
\end{equation}
and finally the \emph{polynomial improving bound}
\begin{equation}\label{imp}
\Big\| \E_{n \in [N]} (\Lambda(n)+|\Lambda_N(n)|) |g(\cdot-P(n)+n)| \Big\|_{\ell^{p'}(\mathbb{Z})} \lesssim_{C_0} N^{d (1/p' - 1/p)} \|g\|_{\ell^p(\mathbb{Z})} 
\end{equation}
for all $u_P< p \leq 2$ and $g \in \ell^p(\Z)$, with $u_P < 2$ an exponent depending only on $P$, and $C>0$ a constant also depending only on $P$.

We shall also require further properties\footnote{Our choice of approximant $\Lambda_N$ will in fact be nonnegative, and although this is not crucial,  it makes it easier to establish the $L^1$ bound \eqref{L1} and the improving bound \eqref{imp}.} of $\Lambda_N$ in the sequel as needed.  

Arguing as in the proof of \cite[Proposition 3.2(i)]{KMT0} (inserting the nonnegative weight $\Lambda$ as necessary), we see that the pointwise convergence claim of Theorem~\ref{main-thm} follows from the ``H\"older variational estimate'' \eqref{variational}, so we focus now on this estimate.  Henceforth we fix $p_1,p_2,p,d,P,r,\lambda$, as well as the finite $\lambda$-lacunary set $\D$. We allow all constants to depend on $p_1, p_2, p, d, P, r, \lambda$ (but not on $\D$). As in \cite[\S 5]{KMT0}, we now select sufficiently large parameters
$$ 1 \lesssim C_0 \lesssim C_1 \lesssim C_2 \lesssim C_3.$$

By a routine application of Calder\'{o}n's transference principle (\cite[Theorem 3.2(ii)]{KMT0}), adapted to this weighted setting), it suffices to prove \eqref{variational} for the integer shift system $(\Z, |\cdot|, x \mapsto x-1)$, endowed with counting measure $|\cdot|$. Thus, our task is now to show that
$$
  \| (\Avg_{N,\Lambda;\Z}(f,g))_{N \in \D} \|_{\ell^p(\Z; \V^r)} \lesssim \|f\|_{\ell^{p_1}(\Z)} \|g\|_{\ell^{p_2}(\Z)}
$$
for all $f \in\ell^{p_1}(\Z)$ and $g \in \ell^{p_2}(\Z)$.  Arguing as in  the proof of \cite[Proposition 3.2(iii)]{KMT0} (inserting the weight $\Lambda$ as needed), it suffices to prove the ``upper half''
\begin{equation}\label{variational-Z}
  \| (\tilde \Avg_{N,\Lambda}(f,g))_{N \in \D} \|_{\ell^p(\Z; \V^r)} \lesssim \|f\|_{\ell^{p_1}(\Z)} \|g\|_{\ell^{p_2}(\Z)}
\end{equation}
of this estimate, where the averaging operators $\tilde \Avg_{N,w}$ were defined in \eqref{avg-upper-def}.

The next step is to replace the von Mangoldt weight $\Lambda$ by the approximant $\Lambda_N$.

\begin{lemma}[From $\Lambda$ to $\Lambda_N$] 
In order to prove~\eqref{variational-Z} (and hence~\eqref{variational}), it suffices to show that
\begin{equation}\label{atilde} \| (\tilde \Avg_{N,\Lambda_N}(f,g))_{N \in \D} \|_{\ell^{p}(\Z; \V^r)} \lesssim_{C_3} \|f\|_{\ell^{p_1}(\Z)} \|g\|_{\ell^{p_2}(\Z)}.
\end{equation}
\end{lemma}

\begin{proof}
Assuming~\eqref{atilde}, from the triangle inequality and the lacunarity of $\D$ we see that~\eqref{variational-Z} reduces to the single-scale estimate
\[ \| \tilde \Avg_{N,\Lambda - \Lambda_N}(f,g) \|_{\ell^{p}(\Z)} \lesssim_{C_3} \langle \Log N \rangle^{-2} \| f \|_{\ell^{p_1}(\Z)} \| g \|_{\ell^{p_2}(\Z)}.  \]
for each $N \in \D$.

Using the triangle and H\"older inequalities, the prime number theorem, and the hypothesis~\eqref{L1}, we may bound
\[ \| \tilde \Avg_{N,\Lambda - \Lambda_N}(f,g) \|_{\ell^{p}(\Z)} \lesssim_{C_0} \| f \|_{\ell^{p_1}(\Z)} \| g \|_{\ell^{p_2}(\Z)},  \]
so by interpolation (modifying the exponents $p_1,p_2,p$ as needed) it suffices to prove the $\ell^2 \times \ell^2 \to \ell^1$ bound
\begin{align}\label{eq:singlescale}
\| \tilde \Avg_{N,\Lambda - \Lambda_N}(f,g) \|_{\ell^{1}(\Z)} \lesssim_{A,C_3} \langle \Log N \rangle^{-A} \| f \|_{\ell^2(\Z)} \| g \|_{\ell^2(\Z)}
\end{align}
for any $A>0$.

We claim that it suffices to prove \eqref{eq:singlescale} when $f,g$ are supported on intervals of length $N^d$. 
Write
\begin{align*}
f=\sum_{i\in \Z}f_i,\quad g=\sum_{i\in \Z}g_i,\quad f_i=f\ind{(iN^d,(i+1)N^d]},\quad g_i=g\ind{(iN^d,(i+1)N^d]}.    
\end{align*}
Let $C=C_P$ be such that $\{P(n)\colon n\in [N]\}$ is contained in an interval of length $CN^d$. Supposing that~\eqref{eq:singlescale} holds whenever $f,g$ are supported on intervals of length $N^d$, by the triangle inequality and Cauchy--Schwarz we have
\begin{align*}
\| \tilde \Avg_{N,\Lambda - \Lambda_N}(f,g) \|_{\ell^{1}(\Z)} &\lesssim_{A,C_3} \langle \Log N \rangle^{-A} \sum_{\substack{i,j\in \Z\\|i-j|\leq C+1}} \| f_i \|_{\ell^2(\Z)} \| g_j \|_{\ell^2(\Z)}\\
&\lesssim_C \langle \Log N \rangle^{-A}\max_{k\in \Z}\sum_{i\in \Z}\| f_i \|_{\ell^2(\Z)} \| g_{i+k} \|_{\ell^2(\Z)}\\
&\leq \langle \Log  N\rangle^{-A} \max_{k\in \Z}\left(\sum_{i\in \Z}\| f_i \|_{\ell^2(\Z)}^2\right)^{1/2} \left(\sum_{i\in \Z}\| g_{i+k} \|_{\ell^2(\Z)}^2\right)^{1/2}\\
&\leq \langle \Log N \rangle^{-A}\|f\|_{\ell^2(\Z)}\|g\|_{\ell^2(\Z)}.  
\end{align*}
Assume henceforth that $f,g$ are supported on intervals of length $N^d$
in~\eqref{eq:singlescale}. By translation, we can further assume that $g$ is supported on $[N^d]$.

By duality, for some function $h\in \ell^{\infty}(\Z)$ with $|h|\leq 1$ we have
\begin{align}\label{eq:dual}
\| \tilde \Avg_{N,\Lambda - \Lambda_N}(f,g) \|_{\ell^{1}(\Z)} =\Big| \sum_{x\in \mathbb{Z}} h(x) \tilde A_{N,\Lambda-\Lambda_N}(f,g)(x)\Big|  = \Big| \sum_{x\in \mathbb{Z}} f(x) \tilde A_{N,\Lambda-\Lambda_N}^*(h,g)(x)\Big|,
\end{align} 
where 
\[ \tilde \Avg_{N,\Lambda-\Lambda_N}^*(h,g)(x) \coloneqq \mathbb{E}_{n\in [N]} (\Lambda-\Lambda_N)(n) h(x+n) g(x-P(n) +n)\]
is one of the adjoint averaging operators. By Cauchy--Schwarz, the desired estimate~\eqref{eq:singlescale} follows from~\eqref{eq:dual} if we show that  
\begin{align*}
\|\tilde A_{N,\Lambda-\Lambda_N}^*(h,g)(x)\|_{\ell^2(\Z)} \lesssim_{A,C_3} \langle \Log N \rangle^{-A} \|g\|_{\ell^2(\Z)}. \end{align*}
By~\eqref{imp} and the triangle inequality, for all $u_P<q\leq 2$ we have
\begin{align}\label{eq:int1}
\|\tilde A_{N,\Lambda-\Lambda_N}^*(h,g)\|_{\ell^{q'}(\Z)}\leq \|\tilde A_{N,\Lambda-\Lambda_N}^*(1,|g|)\|_{\ell^{q'}(\Z)}\lesssim N^{d(1/q'-1/q)}\|g\|_{\ell^q(\Z)}.    
\end{align}
On the other hand,~\cite[Theorem 4.1]{joni} (i.e,~\eqref{unif-est}), the assumption on the support of $g$, and the hypotheses \eqref{lambda-approx}, \eqref{tau-n}, we have
\begin{align}\label{eq:int2}
\| \tilde \Avg^*_{N,\Lambda - \Lambda_N}(h,g) \|_{\ell^{1}(\Z)} \lesssim_{A,C_3} \langle \Log N \rangle^{-A} N^d \| g \|_{\ell^\infty(\Z)}
\end{align}
for any $A>0$. Interpolating~\eqref{eq:int1} and~\eqref{eq:int2}, the claim~\eqref{eq:singlescale} follows. 
\end{proof}

With this lemma, we can now pass to the approximant $\Lambda_N$.

We are left with showing~\eqref{atilde}. Note from \eqref{L1} and the triangle and H\"older inequalities that $\tilde \Avg_{N,\Lambda_N}$ is bounded from $\ell^{p_1}(\Z) \times \ell^{p_2}(\Z)$ to $\ell^{p}(\Z)$ whenever $\frac{1}{p_1} + \frac{1}{p_2} = \frac{1}{p}$; the challenge is to estimate all the scales $N$ in $\D$ simultaneously in $\V^r$ norm.  We can restrict attention to scales $N \geq C_3$, since the contribution of the case $N < C_3$ can be handled just from the H\"older and triangle inequalities. The fact that the weight function $\Lambda_N$ now depends on $N$ will not significantly impact the arguments that follow.

As in \cite[\S 5]{KMT0}, we introduce the Ionescu--Wainger parameter
$$ \rho \coloneqq 1/C_1.$$
We use $c$ to denote various small positive constants that can depend on the fixed quantities $p_1, p_2, d, P, r$, but do not depend on $C_0,C_1,C_2,C_3$ (or $\rho$).  As reviewed in Section \ref{fourier-sec}, this allows us to create major arc sets $\M_{\leq l, \leq k}$, $\M_{l,\leq k}$ for $l \in \N$, $k \in \Z$, as well as associated Ionescu--Wainger multipliers $\Pi_{\leq l, \leq k}$, $\Pi_{l,\leq k}$.  As in \cite[(5.8)]{KMT0}, we say that the pair $(l,k)$ has \emph{good major arcs} if
$$ k \leq -C_\rho 2^{\rho l}$$
for some sufficiently large $C_\rho$ depending only on $\rho$.  This condition will always be satisfied in practice, and will ensure that the intervals $[\alpha-2^k, \alpha+2^k]$ that comprise $\M_{\leq l, \leq k}$ in \eqref{major-arcs} are disjoint, thus avoiding any difficulties arising from ``aliasing''.

In Section~\ref{verifying-sec}, we shall establish the following crucial variant of \cite[Theorem 5.12]{KMT0}.

\begin{theorem}[Single scale minor arc estimate]\label{ssmae}  Let $N \geq 1$, $l \in \N$, and suppose that $f,g \in \ell^2(\Z)$ obey one of the following two properties:
\begin{itemize}
    \item[(i)] $\F_\Z f$ vanishes on $\M_{\leq l, \leq -\Log N+l}$;
    \item[(ii)] $\F_\Z g$ vanishes on $\M_{\leq l, \leq -d\Log N + dl}$.
\end{itemize}
Then one has
$$ \| \tilde \Avg_{N, \Lambda_N}(f,g)\|_{\ell^1(\Z)} \lesssim_{C_1} (2^{-cl} + \langle \Log N \rangle^{-cC_1}) \|f\|_{\ell^2(\Z)} \|g\|_{\ell^2(\Z)}.$$
\end{theorem}

As in \cite[(5.22)]{KMT0}, we introduce the scales
$$ l_{(N)} \coloneqq C_0 \Log \Log N$$
and repeat the arguments in \cite[\S 5]{KMT0} all the way to \cite[(5.25)]{KMT0}, inserting the weight $\Lambda_N$ as needed, to reduce to establishing the bound
$$ \| (\tilde \Avg_{N, \Lambda_N}(\Pi_{l_1, \leq -\Log N + l_{(N)}} f, \Pi_{l_2, \leq -d\Log N + dl_{(N)}} g))_{N \in \D; l_1, l_2 \leq l_{(N)}}\|_{\ell^{p_0}(\Z;\V^r)} \lesssim_{C_3} 2^{-\rho l} \|f\|_{\ell^2(\Z)} \|g\|_{\ell^2(\Z)}$$
for all $l_1,l_2 \in \N$ where $l \coloneqq \max(l_1,l_2)$.

Now we fix $l_1,l_2$, and (as in \cite[(5.26)]{KMT0}) introduce the quantity
\begin{equation}\label{u-def} u \coloneqq \lfloor C_2 2^{2\rho l} \rfloor.
\end{equation}
As in \cite[(5.27), (5.28)]{KMT0}, we introduce the frequency-localized functions
\begin{equation}\label{FN-def}
    F^{u,l_1,s_1}_N \coloneqq \begin{cases} \Pi_{l_1, \leq -\Log N + s_1} f - \Pi_{l_1, \leq -\Log N + s_1-1} f, & s_1 > -u \\
   \Pi_{l_1, \leq -\Log N - u} f, & s_1 = -u
   \end{cases}
   \end{equation}
   and
\begin{equation}\label{GN-def}
   G^{u,l_2,s_2}_N \coloneqq \begin{cases} \Pi_{l_2, \leq d(-\Log N+s_2)} g - \Pi_{l_2, \leq d(-\Log N+s_2-1)} g,  & s_2 > -u \\
   \Pi_{l_2, \leq d(-\Log N-u)} g, & s_2 = -u.
   \end{cases}
\end{equation}
for any integers $-u \leq s_1, s_2 \leq l_{(N)}$.  Arguing as in the text up to \cite[Theorem 5.30]{KMT0}, inserting the weight $\Lambda_N$ as necessary, it now suffices to establish the following. 

\begin{theorem}[Variational paraproduct estimates]\label{varp}  Let $l_1, l_2 \in \N$, $l \coloneqq \max(l_1,l_2)$, let $f,g \colon \Z \to \C$ be finitely supported, and define $u$ by \eqref{u-def}.  Let $s_1,s_2 \geq -u$, and then let $F_N \coloneqq F^{u,l_1,s_1}_N$, $G_N \coloneqq G^{u,l_2,s_2}_N$, $\I \coloneqq \I^{l,s_1,s_2}$ be defined respectively by \eqref{FN-def}, \eqref{GN-def}, and
$$ \I \coloneqq \{ N \in \D\colon l, s_1, s_2 \leq l_{(N)} \}.$$
Then
    \begin{multline}\label{all-all-weak}
    \| ( \tilde A_{N,\Lambda_N}( F_N, G_N ))_{N \in \I} \|_{\ell^p(\Z;\V^r)} \\
    \lesssim_{C_3} \langle \max(l,s_1,s_2) \rangle^{O(1)} 2^{O(\rho l)-c \max(l,s_1,s_2) \ind{p_1=p_2=2}}
    \|f\|_{\ell^{p_1}(\Z)} \|g\|_{\ell^{p_2}(\Z)}.
    \end{multline}
\end{theorem}

Repeating the proof of \cite[Proposition 5.33]{KMT0}, inserting the weight $\Lambda_N$ as needed, we see that Theorem \ref{varp} already holds in the ``high-high'' case where $s_1,s_2 > -u$ and $p_1=p_2=2$. Thus we may assume that at least one of the statements $s_1=-u$, $s_2=-u$, or $(p_1,p_2) \neq (2,2)$ holds.

We now begin the arguments in \cite[\S 7]{KMT0}.  We introduce the functions
$$ F \coloneqq \Pi_{l_1, \leq -u}f; \quad G \coloneqq \Pi_{l_2, \leq -u} g$$
and note that
$$ F_N = T^{l_1}_{\varphi_N} F; \quad G_N = T^{l_2}_{\tilde \varphi_N} G$$
where
\begin{equation}\label{varphi-def}
    \varphi_N(\xi) \coloneqq \begin{cases} \eta( 2^{\Log N-s_1} \xi) - \eta( 2^{\Log N-s_1+1} \xi),  & s_1 > -u \\
   \eta( 2^{\Log N+u} \xi),  & s_1 = -u
   \end{cases}
   \end{equation}
   and
   \begin{equation}\label{varphip-def}
   \tilde \varphi_N(\xi) \coloneqq \begin{cases} \eta( 2^{d(\Log N-s_2)} \xi) - \eta( 2^{d(\Log N-s_2+1)} \xi), & s_2 > -u \\
   \eta( 2^{d(\Log N+u)} \xi),  & s_2 = -u.
   \end{cases}
   \end{equation}
Repeating the arguments up to \cite[(7.7)]{KMT0}, we thus see that it suffices to show that the tuple
$$ (\tilde A_{N,\Lambda_N}( T^{l_1}_{\varphi_N} F, T^{l_2}_{\tilde{\varphi}_N} G ))_{N \in \I}$$
is ``acceptable'' in the sense that it has an $\ell^{p_0}(\Z;\V^r)$ norm of
$$ \lesssim_{C_3}  \langle \max(l,s_1,s_2)\rangle^{O(1)} 2^{O(\rho l) - c\max(l,s_1,s_2) \ind{p_1=p_2=2}} \|F\|_{\ell^{p_1}(\Z)} \|G\|_{\ell^{p_1}(\Z)} .$$

We introduce the arithmetic symbol $m_{\hat \Z^\times} \colon (\Q/\Z)^2 \to \C$ by the formula
\begin{equation}\label{mhat-def}
m_{\hat \Z^\times}\bigg( \frac{a_1}{q} \mod 1, \frac{a_2}{q} \mod 1 \bigg) = \E_{n \in (\Z/q\Z)^\times} e\left( \frac{a_1 n + a_2 P(n)}{q} \right)
\end{equation}
for any $q \in \Z_+$ and $a_1,a_2 \in \Z$; this differs from the corresponding symbol $m_{\hat \Z}$ in \cite{KMT0} by restricting $n$ to the primitive residue classes of $\Z/q\Z$ rather than all residue classes, which is a key effect of weighting by $\Lambda$.  It is easy to see from the Chinese remainder theorem that $m_{\hat \Z^\times}$ is well-defined, in the sense that replacing $a_1, a_2, q$ by $k a_1, ka_2, kq$ for any positive integer $k$ does not affect the right-hand side of \eqref{mhat-def}. Given any Schwartz function $m \colon \R^2 \to \C$, we then define the twisted bilinear multiplier operator $\B^{l_1,l_2,m_{\hat \Z^\times}}_{m}(f,g)$ for rapidly decreasing $f,g \colon \Z \to \C$ by the formula
\begin{multline*}
    \B^{l_1,l_2,m_{\hat \Z^\times}}_{m}(f,g)(x) \coloneqq \sum_{\alpha_1 \in (\Q/\Z)_{l_1}, \alpha_2 \in (\Q/\Z)_{l_2}} m_{\hat \Z^\times}(\alpha_1,\alpha_2) \\
\quad \times \int_{\R^2} m(\xi_1,\xi_2) \F_\Z f(\alpha_1+\xi_1) \F_\Z g(\alpha_2+\xi_2) e(-x(\alpha_1+\alpha_2+\xi_1+\xi_2))\ d\xi_1 d\xi_2.
\end{multline*}
As in \cite[(7.9)]{KMT0}, we also introduce the continuous symbol $\tilde m_{N,\R} \colon \R^2 \to \C$ by the formula
$$ \tilde m_{N,\R}(\xi_1,\xi_2) \coloneqq \int_{1/2}^1 e(\xi_1 N t + \xi_2 P(N t))\ dt$$
and also the cutoff functions
$$ \eta_{\leq k}(\xi) \coloneqq \eta(\xi/2^k)$$
for any integer $k$ and frequency $\xi \in \R$, where $\eta \colon \R \to [0,1]$ is a fixed smooth even function supported on $[-1,1]$ that equals one on $[-1/2,1/2]$.

In Section~\ref{verifying-sec}, we will prove the following analogue of \cite[Proposition 7.13]{KMT0}.

\begin{proposition}[Major arc approximation of $\tilde A_{N,\Lambda_N}$]\label{mod-approx}  For any $N \geq 1$ and $s \in \N$ with $-\Log N+s \leq -u$, we have
    \begin{multline}\label{norma}
    \left\| \tilde A_{N,\Lambda_N}\left( \Pi_{l_1, \leq -\Log N+s} \tilde F, \Pi_{l_2, \leq -d\Log N+ds} \tilde G \right) - \B^{l_1, l_2, m_{\hat \Z^\times}}_{(\eta_{\leq -\Log N+s} \otimes \eta_{\leq -d\Log N+ds})\tilde m_{N,\R} }(\tilde F,\tilde G) \right\|_{\ell^p(\Z)} \\
    \quad \lesssim_{C_3} 2^{O(\max(2^{\rho l},s))} \exp(-\Log^c N)
    \| \tilde F \|_{\ell^{p_1}(\Z)} \|\tilde G\|_{\ell^{p_2}(\Z)}
    \end{multline}
    for all $\tilde F \in \ell^{p_1}(\Z), \tilde G \in \ell^{p_2}(\Z)$.
    \end{proposition}

    This is a slightly weaker type of bound than the corresponding result in \cite{KMT0}, as the polynomial gain of $N^{-1}$ has been reduced to the quasipolynomial gain of $\exp(-\Log^c N)$. However, this is still good enough to dominate the $2^{O(\max(2^{\rho l},s))}$ terms, since from \cite[(7.1)]{KMT0} one has
    \begin{equation}\label{n-large}
        N \geq \max(2^{2^{\max(l,s_1,s_2)/2}}, C_3)
    \end{equation}
    for all $N \in \I$.  Because of this, we can repeat the Fourier-analytic arguments in \cite[\S 7]{KMT0} down to \cite[Theorem 7.23]{KMT0} with the obvious changes, and reduce to showing the acceptability of the small-scale model tuple
    \begin{equation}\label{other-small}
        \Big( \int_{1/2}^1 \B^{l_1,l_2,m_{\hat \Z^\times}}_{m_*}( \T^{l_1}_{\varphi_{N,t}} F, \T^{l_2}_{\tilde \varphi_{N,t}} G)\ dt \Big)_{N \in \I_{\leq}}
       \end{equation}
   and the large-scale model tuple
   \begin{equation}\label{all-all-large}
    \Big(\int_{1/2}^1 \B_{1 \otimes m_{\hat \Z^\times}}(\T_{\varphi_{N,t} \otimes 1} F_\A,\T_{\tilde \varphi_{N,t} \otimes 1} G_\A) \Big)_{N \in \I_{>}}
   \end{equation}
   where
   \begin{itemize}
   \item[(i)] $\I_{\leq} \coloneqq \{ N \in \I\colon N \leq 2^{2^u} \}$ and $\I_{>} \coloneqq \{ N \in \I\colon N > 2^{2^u} \}$;
   \item[(ii)] $m_*(\xi_1,\xi_2) \coloneqq \eta_{\leq -2u}(\xi_1) \eta_{\leq -2du}(\xi_2)$;
   \item[(iii)] $\varphi_{N,t}(\xi) \coloneqq \varphi_N(\xi) e(Nt\xi)$, $\tilde \varphi_{N,t}(\xi) \coloneqq \varphi_N(\xi) e(P(Nt)\xi)$;
   \item[(iv)] The adelic model functions $F_\A \in L^{p_1}(\A_\Z)$, $G_\A \in L^{p_2}(\A_\Z)$ are defined by the formulae
   \begin{equation}\label{vecf-def}
    F_\A(x,y) \coloneqq \sum_{\alpha_1 \in (\Q/\Z)_{l_1}} \int_\R \eta_{\leq -2^{u-1}}(\xi_1) \F_\Z F(\alpha_1 + \xi_1) e( -(\xi_1,\alpha_1) \cdot (x,y) ) \ d\xi_1
    \end{equation}
    and
    \begin{equation}\label{vecg-def}
    G_\A(x,y) \coloneqq \sum_{\alpha_2 \in (\Q/\Z)_{l_2}} \int_\R \eta_{\leq -2^{u-1}}(\xi_2) \F_\Z G(\alpha_2 + \xi_2) e( -(\xi_2,\alpha_2) \cdot (x,y) ) \ d\xi_2
    \end{equation}
    for $x \in \R, y \in \hat \Z$.
   \end{itemize}
   
We can then repeat the integration by parts arguments in the remainder of \cite[\S 7]{KMT0} (replacing $m_{\hat \Z}$ by $m_{\hat \Z^\times}$) and reduce to establishing the small-scale model estimate
    \begin{equation}\label{other-small-2}
    \begin{split}
&   \left\|\left( \B^{l_1,l_2,m_{\hat \Z^\times}}_{m_*}( \T^{l_1}_{\varphi_{N,t,j_1}} F, \T^{l_2}_{\tilde \varphi_{N,t,j_2}} G) \right)_{N \in \I_{\leq}} \right\|_{\ell^p(\Z;\V^r)}\\ &\quad \lesssim_{C_3} \langle \max(l,s_1,s_2) \rangle^{O(1)} 2^{O(\rho l)-c l \ind{p_1=p_2=2}}  \|F\|_{\ell^{p_1}(\Z)} \|G\|_{\ell^{p_2}(\Z)}.
\end{split}
    \end{equation}
and the large-scale model estimate
\begin{equation}\label{all-all-large-2}
\begin{split}
&\left\| \left(\B_{1 \otimes m_{\hat \Z^\times}}(\T_{\varphi_{N,t,j_1} \otimes 1} F_\A,\T_{\tilde \varphi_{N,t,j_2} \otimes 1} G_\A) \right)_{N \in \I_{>}}\right\|_{L^p(\A_\Z; \V^r)} \\
&\quad \lesssim_{C_3} \langle \max(l,s_1,s_2) \rangle^{O(1)} 2^{O(\rho l)-c l \ind{p_1=p_2=2}} \|F_\A\|_{L^{p_1}(\A_\Z)} \|G_\A\|_{L^{p_2}(\A_\Z)}.
\end{split}
\end{equation}
whenever $1/2 \leq t \leq 1$ and $j_1,j_2 \in \{-1,0,+1\}$ are such that
\begin{equation}\label{no-sing}
(s_1,j_1), (s_2,j_2) \neq (-u,-1),
\end{equation}
where
\begin{equation}\label{pntj}
    \varphi_{N,t,j_1}(\xi_1) \coloneqq (2^{-s_1} N \xi_1)^{j_1} \varphi_{N,t}(\xi_1)
    \end{equation}
    and
    \begin{equation}\label{tpntj}
    \tilde \varphi_{N,t,j_2}(\xi_2) \coloneqq (2^{-ds_2} N^d \xi_2)^{j_2} \tilde \varphi_{N,t}(\xi_2).
    \end{equation}
To prove the small-scale argument \eqref{all-all-large-2}, we use the two-dimensional Radamacher--Menshov inequality \cite[Corollary 8.2]{KMT0} by repeating the arguments of \cite[\S 8]{KMT0} (replacing $m_{\hat \Z}$ by $m_{\hat \Z^\times}$), reducing matters to establishing the following single-scale estimate.

\begin{lemma}[Single-scale estimate]\label{ssu}  If $\tilde F \in \ell^{p_1}(\Z), \tilde G \in \ell^{p_2}(\Z)$ have Fourier support on ${\mathcal M}_{l_1, \leq -3u}$ and
    ${\mathcal M}_{l_2, \leq -3du}$ respectively, then
    $$    \| \B^{l_1,l_2,m_{\hat \Z^\times}}_{m_*} (\tilde F, \tilde G) \|_{\ell^p(\Z)} \lesssim_{C_3} 2^{-cl \ind{p_1=p_2=2}} \|\tilde F\|_{\ell^{p_1}(\Z)} \|\tilde G\|_{\ell^{p_2}(\Z)}.$$
\end{lemma}

But this can be proven by repeating the proof of \cite[Lemma 8.6]{KMT0}, using Proposition \ref{mod-approx} in place of \cite[Proposition 7.13]{KMT0}; the replacement of $m_{\hat \Z}$ with $m_{\hat \Z^\times}$ makes no difference here, and the slight reduction in strength of Proposition \ref{mod-approx} from a polynomial gain in $N$ to a quasipolynomial gain in $N$ is similarly manageable.

It remains to establish the large-scale estimate \eqref{all-all-large-2}.  We repeat the arguments in \cite[\S 9]{KMT0}, replacing $m_{\hat \Z}$ by $m_{\hat \Z^\times}$, and noting that $\B_{1 \otimes m_{\hat \Z^\times}}$ is the tensor product of the identity and the bilinear operator $\Avg_{\hat \Z^\times}$ on the profinite integers defined for $f \colon \Z/Q\Z \to \C$, $g \colon \Z/Q\Z \to \C$ for any $Q$ (which one can also view as functions on $\hat \Z$ in the obvious fashion) by the formula
$$ \Avg_{\hat \Z^\times}(f,g)(x) \coloneqq \E_{n \in (\Z/Q\Z)^\times} f(x+n) g(x+P(n)).$$
These arguments reduce matters to establishing the following analogue of \cite[Theorem 9.9]{KMT0}.

\begin{theorem}[Arithmetic bilinear estimate]\label{bile}  Let $l \in \N$, and let $f, g \in L^2(\hat \Z)$ obey one of the following hypotheses:
    \begin{itemize}
        \item[(i)] $\F_{\hat \Z} f$ vanishes on $(\Q/\Z)_{\leq l}$;
        \item[(ii)] $\F_{\hat \Z} g$ vanishes on $(\Q/\Z)_{\leq l}$.
    \end{itemize}
    Then for any $1 \leq r < \frac{2d}{d-1}$ one has
    $$ \| \Avg_{\hat \Z^\times}(f,g) \|_{L^r(\hat \Z)} \lesssim_{C_3,r} 2^{-c_r l} \| f\|_{L^2(\hat \Z)} \| g \|_{L^2(\hat \Z)}$$
\end{theorem}

Repeating the arguments in \cite[\S 10]{KMT0} up to \cite[(10.3), (10.4)]{KMT0}, using $\Avg_{\hat \Z^\times}$ in place of $\Avg_{\hat \Z}$, and Theorem \ref{ssmae} in place of \cite[Theorem 5.12]{KMT0}, we see that it suffices to establish the $p$-adic bound
\begin{equation}\label{avg-1}
    \| \Avg_{\Z_p^\times} \|_{L^2(\Z_p) \times L^2(\Z_p) \to L^q(\Z_p)} \lesssim_q 1
    \end{equation}
    for all primes $p$, together with the improvement
    \begin{equation}\label{avg-2}
    \| \Avg_{\Z_p^\times} \|_{L^2(\Z_p) \times L^2(\Z_p) \to L^q(\Z_p)} \leq 1
    \end{equation}
    whenever $1 \leq q < \frac{2d}{d-1}$ and $p$ is sufficiently large depending on $q$, where the averaging operator $\Avg_{\Z_p^\times}$ is defined as
$$ \Avg_{\Z_p^\times}(f,g)(x) \coloneqq \E_{n \in \Z_p^\times} f(x+n) g(x+P(n)).$$
Because $\Z_p^\times$ has density $\frac{p-1}{p}$ in $\Z_p$, we have the pointwise bound
\begin{equation}\label{avg-bound}
|\Avg_{\Z_p^\times}(f,g)(x)| \leq \frac{p}{p-1} \Avg_{\Z^p}(|f|, |g|)(x)
\end{equation}
from the triangle inequality, where
$$ \Avg_{\Z_p}(f,g)(x) \coloneqq \E_{n \in \Z_p} f(x+n) g(x+P(n)).$$
Hence \eqref{avg-1} is immediate from \cite[(10.3)]{KMT0}.  It remains to establish \eqref{avg-2}.  As in \cite[\S 10]{KMT0}, we may assume $2 < q < \frac{2d}{d-1}$ and $\|f\|_{L^2(\Z_p)} = \|g\|_{L^2(\Z_p)} = 1$ with $f,g$ nonnegative, in which case our task is to show that
$$ \E_{n \in \Z_p} |\Avg_{\Z_p^\times}(f,g)(x)|^q \leq 1.$$

Applying \eqref{avg-bound} and  the bound $\|\Avg_{\Z_p}(|f|,|g|)\|_{L^q(\mathbb{Z}_p)}\leq 1$ from \cite[\S 10]{KMT0} would cost a factor of $(\frac{p}{p-1})^q$, which is not acceptable here (the product $\prod_p \frac{p}{p-1}$ diverges).  Instead, we follow the arguments in \cite[\S 10]{KMT0}, decomposing $f = a + f_0$, $g = b + g_0$, where $0 \leq a,b \leq 1$, $f_0, g_0$ have mean zero, and the ``energies''
$$ E_f \coloneqq \| f_0 \|_{L^2(\Z_p)}^2; \quad E_g \coloneqq \| g_0 \|_{L^2(\Z_p)}^2$$
obey $0 \leq E_f, E_g \leq 1$ and
$$ |a| = (1 - E_f)^{1/2}; \quad |b| = (1  - E_g)^{1/2}.$$

In the case of $\Avg_{\Z_p}$, we clearly have
$$ \Avg_{\Z_p}(a,b) = ab; \quad \Avg_{\Z_p}(f_0,b) = 0$$
(was observed in~\cite[\S 10]{KMT0}) so that by linearity we have
$$ \Avg_{\Z_p}(f,g) = ab + \Avg_{\Z_p}(f,g_0).$$
For the averaging operator $\Avg_{\Z_p^\times}$ the situation is slightly more complicated; we have
$$ \Avg_{\Z_p^\times}(a,b) = ab; \quad \Avg_{\Z_p^\times}(f_0,b) = -\frac{p}{p-1} b h$$
where $h \colon \Z_p \to \R$ is the function
$$ h(x) \coloneqq \E_{n \in \Z_p} f_0(x+n) \ind{p\mid n}.$$
Since $f_0$ has mean zero, $h$ has mean zero as well.  Furthermore, from Young's convolution inequality one has the bounds
\begin{equation}\label{l2-bounds}
    \|h\|_{L^2(\Z_p)} \leq \|f_0\|_{L^2(\Z_p)} \|\ind{p \mid n} \|_{L^1(\Z_p)} = p^{-1} E_f^{1/2}; \quad \|h\|_{L^q(\Z_p)} \leq \|f_0\|_{L^2(\Z_p)} \|\ind{p \mid n} \|_{L^r(\Z_p)} = p^{-1/2-1/q} E_f^{1/2}
\end{equation}
where $1/q + 1 = 1/2 + 1/r$.

We now have the decomposition
$$ \Avg_{\Z_p^\times}(f,g) = ab + \Avg_{\Z_p^\times}(f,g_0) - \frac{p}{p-1} b h$$
and hence by Taylor the expansion $(x+y)^q=x^q+qx^{q-1}y+O(q^2x^{q-2}y^2)$ (as in \cite[\S 10]{KMT0}) we have
\begin{align*}
     |\Avg_{\Z_p^\times}(f,g)|^q &= |ab|^q + q |ab|^{q-1} (\Avg_{\Z_p^\times}(f,g_0) - \frac{p}{p-1} b h) \\
     \quad & + O_q( |\Avg_{\Z_p^\times}(f,g_0)|^2 + |\Avg_{\Z_p^\times}(f,g_0)|^q + |h|^2 + |h|^q ).
\end{align*}
Since $a,b\in [0,1]$, we can bound $|ab|^q\leq |ab|^2 = (1-E_f)(1-E_g)$. Furthermore, $\frac{p}{p-1} b h$ has mean zero, and $\Avg_{\Z_p^\times}(f,g_0)$ has a mean of at most $\| \Avg_{\Z_p^\times}(f_0,g_0)\|_{L^1(\Z_p)}$ since $\Avg_{\Z_p^\times}(a,g_0)$ has mean zero.  We conclude that
\begin{align*}
    \|\Avg_{\Z_p^\times}(f,g)\|_{L^q(\Z_p)}^q &\leq (1-E_f) (1-E_g) \\
    &\quad + O_q( \| \Avg_{\Z_p^\times}(f_0,g_0) \|_{L^1(\Z_p)} + \| \Avg_{\Z_p^\times}(f,g_0) \|_{L^2(\Z_p)}^2 + \| \Avg_{\Z_p^\times}(f,g_0) \|_{L^q(\Z_p)}^q\\
    &\quad\quad + p^{-2} E_f + p^{-q/2-1} E_f^{q/2}).
\end{align*}
By arguing as in \cite[\S 10]{KMT0} (using Theorem \ref{ssmae} in place of \cite[Theorem 5.12]{KMT0}), we see that if $l$ is any large integer and $p$ is sufficiently large depending on $q$, we have the estimates
\begin{align*}
    \| \Avg_{\Z_p^\times}(f_0,g_0) \|_{L^1(\Z_p)} &\lesssim 2^{-c_q l} E_f^{1/2} E_g^{1/2} \\
    \| \Avg_{\Z_p^\times}(f,g_0) \|_{L^2(\Z_p)}^2 &\lesssim 2^{-c_q l} E_g \\
    \| \Avg_{\Z_p^\times}(f,g_0) \|_{L^q(\Z_p)}^q &\lesssim 2^{-c_q l} E_g^{q/2}
\end{align*}
for some $c_q>0$ depending only on $q$, and hence by the arithmetic mean-geometric mean inequality and the hypothesis $q > 2$ we have
\begin{align*} 
\|\Avg_{\Z_p^\times}(f,g)\|_{L^q(\Z_p)}^q &\leq (1-E_f) (1-E_g) + O_q( (2^{-c_q l} + p^{-2}) (E_f+E_g) )\\
&\leq (1-E_f) (1-E_g) + O_q( (2^{-c_q l} + p^{-2})),
\end{align*}
and the right-hand side is bounded by $1$ for $l$ and $p$ large enough, as required.

To summarize, in order to complete the proof of Theorem \ref{main-thm}, we need to select an approximant $\Lambda_N$ to the weight $\Lambda$ at each scale $N$ that obeys the estimates \eqref{lambda-approx}, \eqref{tau-n}, \eqref{L1}, \eqref{imp}, as well as the single scale minor arc estimate in Theorem \ref{ssmae} and the major arc approximation in Proposition~\ref{mod-approx}.  This will be the focus of the next sections.

\section{Approximants to the von Mangoldt function}

As seen in the previous section, the arguments rely on using an approximant $\Lambda_N$ to the von Mangoldt function $\Lambda$ at scale $N$.  There are several plausible candidates for such approximants, including

\begin{itemize}
\item[(i)] $\Lambda$ itself.
\item[(ii)] A \emph{Cram\'er (or Cram\'er--Granville) approximant}
$$ \Lambda_{\Cramer, w}(n) \coloneqq \frac{W}{\varphi(W)} \ind{(n,W)=1}$$
where
$$ W \coloneqq \prod_{p \leq w} p$$
and $w \geq 1$ is a parameter.
\item[(iii)] A \emph{Heath-Brown approximant}
\begin{equation}\label{hb-def}  
    \Lambda_{\HB,Q}(n) \coloneqq \sum_{q < Q} \frac{\mu(q)}{\varphi(q)} c_q(n) 
\end{equation}
where $Q \geq 1$ is a parameter, and $c_q(n)$ are the Ramanujan sums
\begin{equation}\label{cq-def}  
    c_q(n) \coloneqq \sum_{r \in (\Z/q\Z)^\times} e(-rn/q). 
\end{equation}
\end{itemize}

Other possibilities for approximants exist, including Goldston--Pintz--Y{\i}ld{\i}r{\i}m type approximants $(\log R) \sum_{\ell\mid n} \mu(\ell) \eta(\log \ell/\log R)$ and
$(\log R) (\sum_{\ell\mid n} \mu(\ell) \eta(\log \ell/\log R))^2$ for suitable level parameters $R$ and smooth cutoffs $\eta$, Selberg sieve approximants $(\sum_{\ell\mid n} \lambda_{\ell})^2$, or adjustments to several of the previous approximants by a correction term arising from a Siegel zero, but we will not discuss these other options further here.

The choice (i) (i.e., setting $\Lambda_N \coloneqq \Lambda$) is tempting, particularly in view of recent advances in quantitative understanding of functions such as $\Lambda$ in \cite{TT}, \cite{leng-equidistribution}.  However, it turns out that the presence of a Siegel zero would distort the asymptotics of $\Lambda$ to such an extent that the desired approximation in Proposition \ref{mod-approx} no longer holds with quasipolynomial error terms in $N$, which turns out to significantly complicate the analysis (particularly in the small-scale regime, in which one has to modify the Radamacher--Menshov type arguments significantly).  See Section \ref{Remarks-sec} for further discussion.

The choice (ii) has the advantage of being nonnegative, reasonably well controlled in $\ell^\infty$, and also relatively easy to control in Gowers uniformity norms, and so we shall take such a choice for our approximant $\Lambda_N$; specifically we will set
\begin{equation}\label{lambdan-def}
    \Lambda_N = \Lambda_{\Cramer, \exp(\Log^{1/C_0} N)}.
\end{equation}
However, there is one aspect in which this approximant $\Lambda_N(n)$ is not ideal: it is not exactly equal to a ``Type I sum'' $\sum_{\ell\mid n} \lambda_{\ell}$, where $\lambda_{\ell}$ are weights supported on relatively small values of $d$.  The Heath-Brown approximants $\Lambda_{\HB,Q}$ introduced in (iii) are precisely Type I sums, and so we will switch to those approximants at a certain point in the proof.

In order to achieve these goals, we will need to collect some basic facts about the Cram\'er approximants $\Lambda_{\Cramer, w}$ and the Heath-Brown approximants $\Lambda_{\HB,Q}$, which may be of independent interest. 

\subsection{Bounds on the Cram\'er approximant}

We begin with the Cram\'er approximant.  First we record an easy uniform bound.

\begin{lemma}[Uniform bound on Cram\'er model]\label{cramer-uniform}  If $w \geq 1$, then
    $$ 0 \leq \Lambda_{\Cramer, w}(n) \lesssim \langle \Log w\rangle $$
for all $n \in \Z$.
\end{lemma}

\begin{proof}
    This is immediate from the Mertens theorem bound 
    $$\frac{W}{\varphi(W)} = \prod_{p \leq w} \frac{p}{p-1} \lesssim \langle \Log w\rangle.$$
\end{proof}

The Cram\'er approximant is not easily expressible as an exact Type I sum once $w$ is reasonably large (in particular, larger than $\Log N$), but thanks to the fundamental lemma of sieve theory, it can be approximated by such a sum.

\begin{lemma}[Fundamental lemma of sieve theory]\label{fund-lemma} If $2 \leq w \leq y \leq N^{1/10}$, then there exist weights $\lambda^\pm_{\ell} \in [-1,1]$, supported on $1\leq \ell \leq y$, such that
    $$ \sum_{\ell\mid n} \lambda^-_{\ell} \leq \frac{\varphi(W)}{W} \Lambda_{\Cramer,w}(n) \leq \sum_{\ell\mid n} \lambda^+_{\ell}$$
    for all $n$, and also
    $$ \E_{n \in I} \sum_{\ell\mid n} \lambda^\pm_{\ell} = \frac{\varphi(W)}{W} (1 + O( \exp( - \log y/ \log w   ))$$
    for any interval $I$ of length $N$. In particular,
    $$ \E_{n \in I} \left|\Lambda_{\Cramer,w}(n) - \frac{W}{\varphi(W)} \sum_{\ell\mid n} \lambda^\pm_{\ell}\right| \lesssim \exp( - \log y / \log w ).$$
\end{lemma}

\begin{proof}  This follows easily from \cite[Lemma 6.3]{Iwaniec}.
\end{proof}

The fundamental lemma can then be used to give many good estimates for the Cram\'er model.

\begin{proposition}[Linear equations in the Cram\'er model]\label{lineq}  Let $t,m \geq 1$ be integers, and let $N \geq 100$.  Let $\Omega\subset [-N,N]^d$ be convex, and let $\psi_1,\dots,\psi_t \colon \Z^m \to \Z$ be linear forms
    $$ \psi_i(\vec n) = \vec n \cdot \dot \psi_i + \psi_i(0) $$
for some $\dot \psi_i \in \Z^m$ and $\psi_i(0) \in \Z$.  Assume that the linear coefficients $\dot \psi_1,\dots,\dot \psi_t \in \Z^m$ are all pairwise linearly independent and have magnitude at most $\exp(\log^{3/5} N)$.  Suppose that  $1 \leq z_i \leq \exp(\Log^{1/10} N)$ for all $i=1,\dots,t$. Then one has
    $$ \sum_{\vec n \in \Omega \cap \Z^m} \prod_{i=1}^t \Lambda_{\Cramer,z_i}(\psi_i(\vec n)) = \mathrm{vol}(\Omega) \prod_p \beta_p + O_{t,m}( N^m \exp(-c \Log^{4/5} N))$$
    for some $c>0$ depending only on $t,m$, where for each $p$, $\beta_p$ is the local factor
    $$ \beta_p \coloneqq \E_{\vec n \in (\Z/p\Z)^m} \prod_{\substack{1 \leq i \leq t\\ p \leq z_i}} \frac{p}{p-1} \ind{\psi_i(\vec n) \neq 0},$$
    where $\psi_i$ is also viewed as a map from $(\Z/p\Z)^m$ to $\Z/p\Z$ in the obvious fashion.  Furthermore, $\beta_p$ obeys the bounds
    \begin{equation}\label{betap} 
        \beta_p = 1 + O_{t,m}(1/p^2)
    \end{equation}
    for all primes $p$ (and $\beta_p=1$ if $p > \max(z_1,\dots,z_t)$).
\end{proposition}

\begin{proof}  This is essentially \cite[Proposition 5.2]{TT} (which relies to a large extent on the fundamental lemma of sieve theory).  Strictly speaking, this proposition only covered the case where the $z_i$ were equal to a single parameter $z$ which was also assumed to be at least $2$, but an inspection of the argument shows that it applies without significant difficulty to variable $z_i$ as well, even if some of the $z_i$ are as small as $1$.  The bound \eqref{betap} follows from \cite[(5.2), (5.5)]{TT} (a slightly weaker bound, which also suffices for our application, can be found in \cite[Lemma 1.3]{gt-linear})
\end{proof}

Specializing to the $t=m=1$ case (and noting that the constant coefficients of $\psi_i$ can be large in Proposition~\ref{lineq}), we immediately obtain

\begin{corollary}[Mean value of Cram\'er]\label{mean}  Let $N \geq 100$ and $1 \leq z \leq \exp(\Log^{1/10} N)$, then
$$ \E_{n \in I} \Lambda_{\Cramer,z}(n) = 1 + O(\exp(-c \Log^{4/5} N))$$
for any interval $I$ of length $N$.
In particular, since $\Lambda_{\Cramer,z}(n)$ is nonnegative, we also have
$$ \E_{n \in I} |\Lambda_{\Cramer,z}(n)| = 1 + O(\exp(-c \Log^{4/5} N)).$$
More generally, if $1 \leq q \leq z$ and $a\ (q)$ is a residue class, then
$$ \E_{n \in I} \Lambda_{\Cramer,z}(n) \ind{n=a\ (q)} = \frac{\ind{(a,q)=1}}{\varphi(q)} + O(\exp(-c \Log^{4/5} N)).$$
\end{corollary}

As a more sophisticated application of Proposition \ref{lineq}, we record the following improvement of \cite[Proposition 1.2]{TT}.

\begin{lemma}[Improved stability of the Cram\'er model]\label{cramer-stable}  If $1 \leq z,w \leq \exp(\Log^{1/10} N)$, for any $d \ge 1$ one has
$$ \| \Lambda_{\Cramer, w} - \Lambda_{\Cramer, z} \|_{U^{d+1}(I)} \lesssim_{d} w^{-c} + z^{-c}$$
for any interval $I$ of length $N$. In particular, by \eqref{udn},
$$ \| \Lambda_{\Cramer, w} - \Lambda_{\Cramer, z} \|_{u^{d+1}(I)} \lesssim_{d} w^{-c} + z^{-c}.$$
In fact, one can take $c = 1/2^{d+1}$ in these estimates.
\end{lemma}

The result in \cite[Proposition 1.2]{TT} had an additional term of $\Log^{-c} N$ on the right-hand side.  The removal of this term was already conjectured in \cite[Remark 5.4]{TT}.

\begin{proof} Without loss of generality we may assume that $z \leq w$. Expanding out the expression $\| \Lambda_{\Cramer, w} - \Lambda_{\Cramer, z} \|_{U^{d+1}(I)}^{2^{d+1}}$ into an alternating sum of $2^{d+1}$ terms, it suffices to show that
$$ \sum_{\epsilon \in \{0,1\}^{d+1}}\, \sum_{n,h_1,\dots,h_{d+1} \in \Z} \prod_{j=1}^{d+1} \Lambda_{\Cramer, w_\epsilon} \ind{I}(n + \epsilon_1 h_1 + \dots + \epsilon_{k+1} h_{k+1}) = (X + O( z^{-1} ) ) N^{d+2}$$
for all choices of parameters $w_\epsilon \in \{w,z\}$, where $\epsilon = (\epsilon_1,\dots,\epsilon_{d+1})$ and $X$ is a quantity that is independent of the choice of parameters $w_\epsilon$.  Applying Proposition \ref{lineq}, the left-hand side is
$$ \mathrm{vol}(\Omega) \prod_p \beta_p + O_d( N^{d+2} \exp(-c \Log^{4/5} N))$$
where $\Omega$ is a certain explicit convex polytope of volume $\beta_\infty N^{d+2}$ for some constant $\beta_\infty$ depending only on $d$, and the local factors $\beta_p$ are defined by the formula
$$ \beta_p \coloneqq \E_{n,h_1,\dots,h_{d+1} \in \Z/p\Z} \prod_{\substack{\epsilon \in \{0,1\}^{d+1}\\ p \leq w_\epsilon}} \frac{p}{p-1} \ind{p \nmid n + \epsilon_1 h_1 + \dots + \epsilon_{k+1} h_{k+1}}.$$
The local factors $\beta_p$ are independent of the $w_\epsilon$ if $p \leq w$ or $p > z$.  Thus, by \eqref{betap}, the product $\prod_p \beta_p$ can be written as $Y(1+O(1/z))$ for some $Y$ that is independent of the $w_\epsilon$ parameters, and the claim follows.
\end{proof}

\subsection{Bounds on the Heath-Brown approximant}

We now turn to the Heath-Brown approximants $\Lambda_{\HB,Q}$.  The nice bounds in $\ell^\infty$ or $\ell^1$ one has in Lemma \ref{cramer-uniform} or Corollary  \ref{mean} are unfortunately not available for this approximant.  However, we have reasonable control in other norms such as $\ell^2$, in large part due to a good Type I representation.

\begin{lemma}[Moment bounds for Heath-Brown approximant]\label{bounds}  For any $Q \geq 1$, one has the Type I representation
\begin{equation}\label{lhb}
     \Lambda_{\HB,Q}(n) = \sum_{\substack{\ell\mid n\\ \ell < Q}} \lambda_d
\end{equation}
for some weights $\lambda_{\ell}$ with
\begin{equation}\label{lambdad}
 \lambda_{\ell} \lesssim \langle \Log Q \rangle.
\end{equation}
In particular, we have the pointwise bound
\begin{equation}\label{lambdaq-bound} 
\Lambda_Q(n) \lesssim  \tau(n,Q) \langle \Log Q \rangle
\end{equation}
where $\tau(n,Q)$ is the truncated divisor function
$$ \tau(n,Q) \coloneqq \sum_{\substack{\ell\mid n\\ \ell<Q}}1.$$
Furthermore, we have the moment bounds
\begin{equation}\label{lambdaq-moment} 
    \E_{n \in [N]} |\Lambda_Q(n)|^k \lesssim_k \langle \Log Q\rangle^{2^k+k}
\end{equation}
for any positive integer $k$ and $N \geq 1$.
\end{lemma}

\begin{proof} Applying the standard identity $c_q(n)=\sum_{\ell\mid (q,n)}\ell \mu(q/\ell)$ and then writing $q=\ell r$, we have
    \begin{align*}
        \Lambda_Q(n) &= \sum_{q<Q} \frac{\mu(q)}{\varphi(q)}\sum_{\ell\mid (q,n)}\ell\mu(q/\ell) \\
        &= \sum_{\substack{\ell\mid n\\ \ell<Q}} \frac{\mu(\ell) \ell }{\varphi(\ell)} \sum_{\substack{r < Q/\ell\\ (\ell,r)=1}} \frac{\mu^2(r)}{\varphi(r)}.
    \end{align*}
    We then take
    $$ \lambda_{\ell} \coloneqq \frac{\mu(\ell) \ell }{\varphi(\ell)} \sum_{\substack{r < Q/\ell\\ (\ell,r)=1}} \frac{\mu^2(r)}{\varphi(r)}.$$
    From Rankin's trick and Mertens's theorem, for any $1\leq d\leq Q$ one has
\begin{align*}
\sum_{\substack{r \leq Q/\ell\\ (d,\ell)=1}} \frac{\mu^2(r)}{\varphi(r)} &\lesssim \sum_{\substack{r\geq 1\\ (\ell,r)=1}} \frac{\mu^2(r)}{\varphi(r) r^{1/\langle \Log Q\rangle}} \\
&\lesssim \prod_{\substack{p\\ p \nmid \ell}} \left(1+\frac{1}{(p-1) p^{1/\langle \Log Q\rangle}}\right) \\
&\lesssim \frac{\varphi(\ell)}{\ell} \prod_p \left(1 + \frac{1}{p^{1+1/\langle \Log Q\rangle}} + O\left(\frac{1}{p^2}\right)\right) \\
&\lesssim  \frac{\varphi(\ell)}{\ell} \langle \Log Q\rangle,
\end{align*}
where we used the Euler product formula and the standard bound $\zeta(\sigma)\sim \frac{1}{\sigma-1}$ for $\sigma>1$ to estimate the product over the primes. This gives \eqref{lambdad}. The bound \eqref{lambdaq-bound} then follows from the triangle inequality.  

Now we turn to \eqref{lambdaq-moment}.   We may assume that $Q \geq 100$, as the claim is trivial otherwise. We allow all implied constants to depend on $k$.
In view of \eqref{lambdaq-bound}, it suffices to establish the bound
    $$ \sum_{n \in [N]} \tau(n,Q)^k \lesssim N \langle \Log Q\rangle^{2^k}.$$
    We expand
    \begin{align*}
        \sum_{n \in [N]} \tau(n,Q)^k = \sum_{n \in [N]} \Bigg(\sum_{\substack{\ell\mid n\\d<Q}}1\Bigg)^k = \sum_{n \in [N]} \sum_{\ell_1,\ldots, \ell_k<Q} 1 = \sum_{\ell_1,\ldots, \ell_k<Q} \frac{N}{[\ell_1,\ldots, \ell_k]}
    \end{align*}
    where $[a_1,\ldots, a_k]$ is the least common multiple of $a_1,\ldots, a_k$. 
    
    Now we apply Rankin's trick.  For $\ell_i<Q$, we have $\ell_i^{1/\langle \Log Q\rangle} = O(1)$, thus
    \begin{align*}
        \E_{n \in [N]} \tau(n,Q)^k &\lesssim \sum_{\ell_1,\ldots, \ell_k} \frac{1}{\ell_1^{1/\log Q}\cdots \ell_k^{1/\langle \Log Q\rangle}[\ell_1,\ldots, \ell_k]}.
    \end{align*}
Factorizing into an Euler product, we conclude that
    \begin{align*}
        \E_{n \in [N]} \tau(n,Q)^k \lesssim \prod_{p}\left(1+\sum_{\substack{a_1,\ldots, a_k\in \{0,1\}\\(a_1,\ldots,a_k)\neq \mathbf{0}}}\frac{1}{p^{1+(a_1+\cdots a_k)/\langle \Log Q\rangle}} +O\left(\frac{1}{p^2}\right)\right)
    \end{align*}
    where $\mathbf{0} \coloneqq (0,\dots,0)$.
    Hence on taking logarithms, it will suffice to show that
    \begin{align*}
        \sum_p \sum_{\substack{a_1,\ldots,a_k\in \{0,1\}\\(a_1,\ldots, a_k)\neq \mathbf{0}}}p^{-1- \frac{a_1+\cdots+ a_k}{\langle \Log Q\rangle}} \leq 2^k \log\log Q + O(1).
    \end{align*}
    From partial summation and the prime number theorem we have
    \begin{align*}
       \sum_{\substack{a_1,\ldots,a_k\in \{0,1\}\\(a_1,\ldots, a_k)\neq \mathbf{0}}}\sum_{p\geq Q}p^{-1- \frac{a_1+\cdots+ a_k}{\langle \log Q\rangle}} &\leq \sum_{\substack{a_1,\ldots,a_k\in \{0,1\}\\(a_1,\ldots, a_k)\neq \mathbf{0}}} \int_{Q}^{\infty}\frac{t^{-1-\frac{a_1+\cdots+a_k}{\langle \Log Q\rangle}}}{\log t}\, dt+O(1)\\
         &\leq 2^k \cdot \int_{Q}^{\infty} t^{-\frac{1}{\langle \log Q\rangle}} \ \frac{dt}{t \log t} +O(1) \lesssim 2^k + O(1).
    \end{align*}

    Moreover, we can use Mertens's theorem to estimate
    \begin{align*}
        \sum_{\substack{a_1,\ldots,a_k\in \{0,1\}\\(a_1,\ldots, a_k)\neq \mathbf{0}}}\sum_{p<Q}p^{-1- \frac{a_1+\cdots+ a_k}{\langle \Log Q\rangle}} \leq 2^k\log \langle \Log Q \rangle +O(1).
    \end{align*}
    Combining these bounds gives the result.
\end{proof}

\subsection{Comparing the Cram\'er and Heath-Brown approximants}

We have a useful comparison theorem between the Cram\'er and Heath-Brown approximants.

\begin{proposition}[Comparison between Cram\'er and Heath-Brown]\label{comparison}  Let $N \geq 1$ and $1 \leq w, Q \leq \exp(\Log^{1/20} N)$, and let $d \geq 1$ be an integer.  Then
    $$ \| \Lambda_{\Cramer,w} - \Lambda_{\HB,Q} \|_{u^{d+1}(I)} \lesssim_d w^{-c} + Q^{-c}$$
    for any interval $I$ of length $N$.
As a consequence, from Lemma \ref{cramer-stable} and the triangle inequality, we also have
$$ \| \Lambda_{\HB,Q_1} - \Lambda_{\HB,Q_2} \|_{u^{d+1}(I)} \lesssim_d Q_1^{-c} + Q_2^{-c}$$
whenever $1 \leq Q_1,Q_2 \leq \exp(\Log^{1/20} N)$.
\end{proposition}

\begin{proof}  We allow all implied constants to depend on $d$.  In view of Lemma \ref{cramer-stable} and the triangle inequality, it suffices to establish the bound
$$ \| \Lambda_{\Cramer,Q} - \Lambda_{\HB,Q} \|_{u^{d+1}(I)} \lesssim Q^{-c}$$
for any interval $I$ of length $N$,
that is to say it suffices to show that
$$ |\E_{n \in I} (\Lambda_{\Cramer,Q}(n) - \Lambda_{\HB,Q}(n)) e(R(n)) | \lesssim Q^{-c}$$
for any polynomial $R(n) = \sum_{j=0}^d \alpha_j (n-n_I)^d$ of degree at most $d$ with some real coefficients $\alpha_j$, where $n_I$ denotes the midpoint of $I$.  By subdividing $I$ into smaller intervals and using the triangle inequality (adjusting the coefficients $\alpha_j$ as necessary), we may assume without loss of generality that
$$ N \sim \exp(\Log^{20} Q).$$
We can then also assume that $Q$ (and hence $N$) are large, as the claim is trivial otherwise.
In particular $\Log N = \Log^{O(1)} Q$, which in practice will permit us to absorb all logarithmic factors of $N$ in the analysis below.

Fix the polynomial $R$.  We may of course assume without loss of generality that
$$ |\E_{n \in I} (\Lambda_{\Cramer,Q}(n) - \Lambda_{\HB,Q}(n)) e(R(n)) |\geq Q^{-1}.$$
Applying Lemma \ref{fund-lemma} (with $w=Q$ and $y = \exp(\Log^{1/10} N)$) as well as Lemma \ref{bounds}, we thus have
$$\Big| \E_{n \in I} \Big(\sum_{\substack{\ell \leq \exp(\Log^{1/10} N)\\ \ell \mid  n}} \lambda_{\ell}\Big) e(R(n))\Big| \geq Q^{-1}$$
for some weights $\lambda_{\ell}$ of size $O(\Log^{O(1)} N) = O(\Log^{O(1)} Q)$.  Applying \cite[Proposition 2.1]{mat1} (after shifting the summation variable by $n_I$), we conclude that the polynomial $R$ is major arc in the sense that there exists an integer $1 \leq q \lesssim Q^{O(1)}$ such that
$$ \| q \alpha_j \|_{\R/\Z} \lesssim Q^{O(1)} / N^j$$
for all $1 \leq j \leq d$. We may assume that $q\geq Q$ by multiplying $q$ by an integer of size $Q$ if necessary. Thus one can write $R(n) = R_0(n) + E(n)$ where $R_0$ is a polynomial of degree at most $d$ that is periodic with period $q$, and the error $E$ satisfies $\sup_{n\in I}|E(n+1)-E(n)|=O( Q^{O(1)}/N)$.  

Set 
$$w \coloneqq q,\quad W\coloneqq \prod_{p<w}p,$$
thus $Q \leq w \lesssim Q^{O(1)}$.  By Lemma \ref{cramer-stable} and the triangle inequality, it will suffice to show that
$$ |\E_{n \in I} (\Lambda_{\Cramer,w}(n) - \Lambda_{\HB,Q}(n)) e(R(n)) | \lesssim Q^{-c}.$$
Breaking up $I$ into intervals $J$ of length $\sqrt{N}$ and using the slowly varying nature of $E(n)$, it suffices to show that
$$ |\E_{n \in J} (\Lambda_{\Cramer,w}(n) - \Lambda_{\HB,Q}(n)) e(R_0(n)) |\lesssim Q^{-c}$$
for any interval $J$ of length $\sqrt{N}$.  

From Corollary \ref{mean} and the $q$-periodicity of $R_0$ we have
$$ \E_{n \in J} \Lambda_{\Cramer,w}(n) e(R_0(n)) = \E_{n \in (\Z/q\Z)^\times} e(R_0(n)) + O(Q^{-c})$$
(in fact, the error term is significantly better than this).  Using the multiplicativity of the Ramanujan sums $c_q(\cdot)$ and the fact that $c_p(n)=(p-1)\ind{n=0\ (p)}-\ind{n\neq 0\ (p)}$, we have
$$ \sum_{\ell \mid q} \frac{\mu(\ell)}{\varphi(\ell)} c_{\ell}(n) = \prod_{p\mid q} \left(1 - \frac{c_p(n)}{p-1}\right) = \ind{(n,q)=1} \frac{q}{\varphi(q)}.$$
We thus have
\begin{align*} \E_{n \in J} \Lambda_{\Cramer,w}(n) e(R_0(n)) = \sum_{\ell\mid q} \frac{\mu(\ell)}{\varphi(\ell)} \E_{n \in [q]} e(R_0(n)) c_{\ell}(n) + O(Q^{-c}).\end{align*}

Note that for any natural numbers $\ell,a,q$ with $\ell\nmid q$, by the geometric sum formula we have
\begin{align*}
\mathbb{E}_{n\in J}c_{\ell}(n)1_{n\equiv a\pmod q}=\sum_{r\in (\mathbb{Z}/\ell\mathbb{Z})^{\times}}\mathbb{E}_{n\in J}e\left(\frac{rn}{\ell}\right)1_{n\equiv a\pmod q}\ll \ell^2/\sqrt{N}.    
\end{align*}
Therefore, from \eqref{hb-def} and the $q$-periodicity of $e(R_0(n))$, we have
$$ \E_{n \in J} \Lambda_{\HB,Q}(n) e(R_0(n)) = \sum_{\substack{\ell\mid q\\ \ell < Q}} \frac{\mu(\ell)}{\varphi(\ell)} \E_{n \in [q]} e(R_0(n)) c_{\ell}(n) + O(Q^{-c})$$
(again, a better error term is available here). Thus, by the triangle inequality, it suffices to show that
$$ \sum_{\substack{\ell\mid q\\ \ell\geq Q}} \frac{\mu^2(\ell)}{\varphi(\ell)} |\E_{n \in [q]} e(R_0(n)) c_{\ell}(n)| \lesssim Q^{-c}.$$
By the divisor bound, $q$ has at most $Q^{o(1)}$ factors,  so it will suffice to establish the bound
$$ |\E_{n \in [q]} e(R_0(n)) c_{\ell}(n)| \lesssim \varphi(\ell) Q^{-c}$$
for each square-free $\ell\mid q$ with $\ell\geq Q$.  By the triangle inequality, it suffices to show that
$$  \sum_{r \in (\Z/\ell\Z)^\times} |\E_{n \in \Z/q\Z} e(R_0(n) - rn/\ell)| \lesssim \varphi(\ell) Q^{-c}.$$
But from the Plancherel identity (or Bessel inequality) and the fact that $\ell\leq q$, one has
$$  \sum_{r \in (\Z/\ell\Z)^\times} |\E_{n \in \Z/q\Z} e(R_0(n) - rn/\ell)|^2 \leq \frac{\ell}{q}\leq 1,$$
and the claim follows from Cauchy--Schwarz (noting from the hypothesis $\ell\geq Q$ that $\varphi(\ell) \gtrsim Q^{1/2}$, say, so that $\varphi(\ell)^{1/2} \lesssim \varphi(\ell) Q^{-1/4}$).
\end{proof}

\section{Verifying the properties of the approximant}\label{verifying-sec}

Recall the definition of $\Lambda_N$ from~\eqref{lambdan-def}. In this section we verify the properties \eqref{lambda-approx}, \eqref{tau-n}, \eqref{L1}, and \eqref{imp} for $\Lambda_N$, and prove Proposition \ref{mod-approx} and Theorem \ref{ssmae} concerning it.

\textbf{Verifying \eqref{lambda-approx}, \eqref{tau-n}, \eqref{L1}.}
The bound \eqref{L1}  follows from Corollary \ref{mean}, while the bound \eqref{tau-n} follows from Lemma~\ref{cramer-uniform}.
The bound \eqref{lambda-approx} follows for instance from\footnote{Strictly speaking, the results in \cite{mstt} were stated only for $C_0=10$, but an inspection of the arguments reveal that they also apply for larger choices of $C_0$.} \cite[Theorem 1.1(ii)]{mstt} (and could also be extracted from the earlier arguments in \cite{mat1}).

\textbf{Verifying  \eqref{imp}.} 
We need the following weighted analogue of \cite[Proposition 6.21]{KMT0}. 

\begin{lemma}[$L^p$ improving]\label{lp-improv}  Let $Q \in \Z[\n]$ be of degree $d\geq 1$. If $2-c_d < p \leq 2$ for some sufficiently small $c_d>0$, then
$$ \Big\|\mathbb{E}_{n\in [N]}(\Lambda(n)+\Lambda_N(n))f(\cdot+Q(n))\Big\|_{\ell^{2}(\Z)}\lesssim_Q N^{d/2 - d/p} \|f\|_{\ell^p(\Z)}$$
and also for the dual exponent $p'=p/(p-1)$ we have
\begin{align}\label{ppp}
\Big\|\mathbb{E}_{n\in [N]}(\Lambda(n)+\Lambda_N(n))f(\cdot+Q(n))\Big\|_{\ell^{p'}(\Z)}\lesssim_Q N^{d/p' - d/p} \|f\|_{\ell^p(\Z)}.
\end{align}
\end{lemma}

The value of $c_d$ here could be explicitly computed, but we do not attempt to optimize it here.  After Lemma~\ref{lp-improv} has been proven, \eqref{ppp} together with the non-negativity of $\Lambda_N$ immediately implies the required estimate~\eqref{imp}.

\begin{proof}  
By interpolation (adjusting $c_d$ as necessary), it suffices to show the second estimate \eqref{ppp}.

For any polynomial $Q(\n) \in \Z[\n]$, we define the averaging operators $\Avg^{Q,0}_N, \ \Avg^Q_N \colon \ell^p(\Z) \to \ell^p(\Z)$ by the formulas
$$ 
 \Avg^{Q,0}_N f(x) \coloneqq \E_{n \in [N]} f(x+Q(n)) \Lambda(n), \; \; \; 
\Avg^Q_N f(x) \coloneqq \E_{n \in [N]} f(x+Q(n)) \Lambda_N(n).$$ First, the operators $\Avg^Q_N, \Avg^{Q,0}_N$ are bounded on every $\ell^p(\Z)$ thanks to \eqref{L1} and the triangle inequality. With this notation, it suffices to show that
\begin{align}\label{ppp1}\begin{split}
\| \Avg^Q_N f \|_{\ell^{p'}(\Z)}  &\lesssim_Q N^{d/p' - d/p} \|f\|_{\ell^p(\Z)},\\
\| \Avg^{Q,0}_N f \|_{\ell^{p'}(\Z)} &\lesssim_Q N^{d/p' - d/p} \|f\|_{\ell^p(\Z)}.
\end{split}
\end{align}
    
We can write $\Avg^Q_N = \Avg^Q_{N, \exp(\Log^{1/C_0} N)}$ where
$$ \Avg^Q_{N,w} f(x) \coloneqq \E_{n \in [N]} f(x+Q(n)) \Lambda_{\Cramer, w}(n).$$
On the one hand, from Lemma \ref{cramer-uniform} and the results in \cite{koselacey} (see also \cite[Proposition 6.21]{KMT0}) we have
\begin{equation}\label{tnw}
 \| \Avg^Q_{N,w} f \|_{\ell^{p'}(\Z)} \lesssim_Q N^{d/p' - d/p} \langle \Log w\rangle \|f\|_{\ell^p(\Z)}
\end{equation}
for any $2-c < p \leq 2$ (where $c>0$ depends on $d$ and can vary from line to line). On the other hand, from Lemma \ref{cramer-stable} we have
\begin{align}\label{eq:lambdaz}
\E_{n \in [N]} (\Lambda_{\Cramer, w} - \Lambda_{\Cramer, z})(n) e(Q(n)) \lesssim_d z^{-c}
\end{align}
for any $1 \leq z \leq w \leq \exp(\Log^{1/C_0} N)$.

By the Plancherel theorem, this implies that
\begin{align*} \| \Avg^Q_{N,w} f - \Avg^Q_{N,z} f \|_{\ell^2(\Z)}&=\left(\int_{0}^1\left|\sum_{x\in \mathbb{Z}}(\Avg^Q_{N,w} f - \Avg^Q_{N,z} f)(x)e(\theta x)\right|^2\, d\theta\right)^{1/2}\\
& \lesssim_d z^{-c} \left(\int_{0}^1\left|\sum_{x\in \mathbb{Z}}f(x)e(\theta x)\right|^2\, d\theta\right)^{1/2}\\
& \lesssim_d z^{-c} \|f\|_{\ell^2(\Z)}.
\end{align*}

Interpolating (and reducing $c$ as necessary), we see that if $2-c \leq p \leq 2$, then
$$ \| \Avg^Q_{N,w} f - \Avg^Q_{N,z} f \|_{\ell^{p'}(\Z)} \lesssim_Q N^{d/p' - d/p} z^{-c} \|f\|_{\ell^p(\Z)}$$
if $1 \leq z \leq w \leq \exp(\Log^{1/C_0} N)$ is such that $w^{1/2} \leq z$.  Summing this bound telescopically for suitable values of $z, w$, we conclude from the triangle inequality that
$$ \| \Avg^Q_{N} f - \Avg^Q_{N,1} f \|_{\ell^{p'}(\Z)} \lesssim_Q N^{d/p' - d/p} \|f\|_{\ell^p(\Z)}.$$
Combining this with the $w=1$ case of \eqref{tnw}, we obtain the first estimate in~\eqref{ppp1}. 

The second estimate in~\eqref{ppp1} follows similarly, except that in the proof we replace~\eqref{tnw} with 
\begin{equation*}
 \| \Avg^{Q,0}_N f \|_{\ell^{p'}(\Z)} \lesssim_Q N^{d/p' - d/p} \langle \Log N\rangle \|f\|_{\ell^p(\Z)}
\end{equation*}
and replace~\eqref{eq:lambdaz} with 
\begin{align*}
\E_{n \in [N]} (\Lambda - \Lambda_{\Cramer, z})(n) e(Q(n)) \lesssim_d z^{-c}
\end{align*}
and use the first estimate in~\eqref{ppp1}.
\end{proof}

\textbf{Proof of Proposition \ref{mod-approx}.} Arguing as in the proof of \cite[Proposition 7.13]{KMT0}, Proposition~\ref{mod-approx} reduces to establishing the symbol estimates
$$
\Big|\frac{\partial^{j_1}}{\partial \xi_1^{j_1}} \frac{\partial^{j_2}}{\partial \xi_2^{j_2}} M_0( (\alpha_1,\xi_1), (\alpha_2,\xi_2))\Big|
\lesssim_{C_3} 2^{O(\max(2^{\rho l},s))} N^{j_1+dj_2} \exp(-\Log^c N)$$
for $0 \leq j_1,j_2 \leq 2$, $\alpha_1 \in (\Q/\Z)_{l_1}$, $\alpha_2 \in (\Q/\Z)_{l_2}$, and $\xi_1 = O( 2^{s}/N)$, $\xi_2 = O( 2^{ds}/N^d)$, where the symbol $M_0$ is defined by the formula
\begin{multline*}
    M_0( (\alpha_1,\xi_1), (\alpha_2,\xi_2)) \\
    \coloneqq \E_{n \in [N]} e(\alpha_1 n + \alpha_2 P(n)) e(\xi_1 n + \xi_2 P(n)) \Lambda_N(n) \ind{n>N/2}  - m_{\hat \Z^\times}(\alpha_1,\alpha_2) \tilde m_{N,\R}(\xi_1,\xi_2).
\end{multline*}
As in the proof of \cite[Proposition 7.13]{KMT0}, the function $n \mapsto e(\alpha_1 n + \alpha_2 P(n))$ is periodic of some period
\begin{equation}\label{qmag}
    q = O_\rho(2^{O(2^{\rho l})}).
\end{equation}
In particular, from \eqref{n-large} one has
$$ q \leq \exp(\Log^{c_0} N)$$
and hence $q$ divides $W$.  So the function $\Lambda_N(n)$ vanishes outside of the primitive residue classes modulo $q$.  Meanwhile, we have
$$ m_{\hat \Z^\times}(\alpha_1,\alpha_2) = \E_{a \in (\Z/q\Z)^\times} e(\alpha_1 a + \alpha_2 P(a)).$$
By the triangle inequality, it thus suffices to show for each $a \in (\Z/q\Z)^\times$ that
\begin{align*}
    &\Big| \frac{\partial^{j_1}}{\partial \xi_1^{j_1}} \frac{\partial^{j_2}}{\partial \xi_2^{j_2}} (\E_{n \in [N]} e(\xi_1 n + \xi_2 P(n)) \Lambda_N(n) \ind{n=a\ (q)} \ind{n>N/2}  - \frac{1}{\varphi(q)} \tilde m_{N,\R}(\xi_1,\xi_2)) \Big| \\
    &\quad \lesssim_{C_3} 2^{O(\max(2^{\rho l},s))} N^{j_1+dj_2} \exp(-\Log^c N).
\end{align*}
Evaluating the derivatives, it suffices to show that
$$ \Big| \sum_{n \in [N] \backslash [N/2]} w(n) \ind{n=a\ (q)} \Lambda_N(n) - \frac{1}{\varphi(q)} \int_{N/2}^N w(t)\ dt \Big| \lesssim_{C_3} 2^{O(\max(2^{\rho l},s))} N^{j_1+2j_2+1} \exp(-\Log^c N),$$
where
$$ w(t) \coloneqq e(\xi_1 t + \xi_2 P(t)) t^{j_1} P(t)^{j_2}.$$
The function $w$ is smooth with a total variation of $O( 2^{O(\max(2^{\rho l},s))} N^{j_1+2j_2})$.  Summing (or integrating) by parts as in \cite[Lemma 2.2(iii)]{mstt}, it suffices to show that
$$\Big| \sum_{n \in I} \big(\ind{n=a\ (q)} \Lambda_N(n) - \frac{1}{\varphi(q)} |I|\big) \Big| \lesssim_{C_3} N  \exp(-\Log^c N)$$
for all intervals $I$ in $[N,2N]$.  But this follows from Corollary \ref{mean}.

\textbf{Proof of Theorem \ref{ssmae}.} The last remaining task is to establish the single-scale estimate in Theorem \ref{ssmae}.  We first recall an application of Peluse--Prendiville theory.

\begin{proposition}[Unweighted inverse theorem]\label{lip}  Let $N \geq 1$ and $0 < \delta \leq 1$, and let $N_0$ be a quantity with $N_0 \sim N^d$.  Let $f,g,h \colon \Z \to \C$ be be supported on $[-N_0,N_0]$ with 
    \begin{equation}\label{fgh-bound}
        \|f\|_{\ell^\infty(\Z)}, \|g\|_{\ell^\infty(\Z)}, \|h\|_{\ell^\infty(\Z)} \leq 1,
    \end{equation}
     obeying the lower bound
    \begin{equation}\label{hyp}
     |\langle \tilde \Avg_{N,1}(f,g), h \rangle| \geq \delta N^d.
     \end{equation}
     Then there exists a function $F \in \ell^2(\Z)$ with 
     \begin{equation}\label{F-bounds}
    \|F\|_{\ell^\infty(\Z)} \lesssim 1; \quad \|F\|_{\ell^1(\Z)} \lesssim N^d
    \end{equation}
    and with $\F_{\Z}F$  supported in the $O(\delta^{-O(1)}/N)$-neighborhood of some rational $a/b \mod 1 \in \Q/\Z$ with $b = O(\delta^{-O(1)})$ such that
    \begin{equation}\label{F-corr}
    |\langle f, F \rangle| \gtrsim \delta^{O(1)} N^d.
    \end{equation}
Here we use the inner product $\langle f, F \rangle \coloneqq \sum_{n \in \Z} f(n) \overline{F(n)}$.
\end{proposition}

\begin{proof}  See \cite[Proposition 6.6]{KMT0}.
\end{proof}

We now transfer this to the weighted setting, under an additional (mild) largeness hypothesis on $\delta$.

\begin{proposition}[Weighted inverse theorem]\label{lip-wt}  Let $N \geq 1$ and $\exp(-\Log^{1/C_0} N) \leq \delta \leq 1$, and let $N_0$ be a quantity with $N_0 \sim N^d$.  Let $f,g,h \colon \Z \to \C$ be be supported on $[-N_0,N_0]$, obeying \eqref{fgh-bound} and the lower bound
    \begin{equation}\label{hyp-wt}
     |\langle \tilde \Avg_{N,\Lambda_N}(f,g), h \rangle| \geq \delta N^d.
     \end{equation}
     Then the conclusions of Proposition \ref{lip} hold.
\end{proposition}

\begin{proof} We may assume that $N$ is sufficiently large depending on the fixed polynomial $P$, as the claim is easy to establish otherwise.

For any $1 \leq z \leq w \leq \exp(\Log^{1/C_0} N)$, we have from
Lemma \ref{cramer-stable}, Lemma \ref{cramer-uniform}, and \cite[Theorem 4.1]{joni} (i.e., \eqref{unif-est}) that
$$|\langle \tilde \Avg_{N,\Lambda_{\Cramer,w} - \Lambda_{\Cramer,z}}(f,g), h \rangle| \lesssim z^{-c} \langle \Log w\rangle N^d.$$
In particular, we have
\begin{align}\label{Cramer-wz}
|\langle \tilde \Avg_{N,\Lambda_{\Cramer,w} - \Lambda_{\Cramer,z}}(f,g), h \rangle| \lesssim z^{-c} N^d.
\end{align}
for $z \in [w/2,w]$; summing dyadically using the triangle inequality, we conclude that
$$|\langle \tilde \Avg_{N,\Lambda_N - \Lambda_{\Cramer,w}}(f,g), h \rangle| \lesssim w^{-c} N^d$$
for any $1 \leq w \leq \exp(\Log^{1/C_0} N)$.  

The weight $\Lambda_{\Cramer,w}$ is not quite of Type I form, so we now aim to swap it with the Heath-Brown weight $\Lambda_{\HB,w}$.  From Lemma \ref{comparison} we have
\begin{align}\label{eq:CramerHB} 
\| \Lambda_{\Cramer,w} - \Lambda_{\HB,w} \|_{u^{d+1}[N]} \lesssim w^{-c}.
\end{align}
We would like to apply \cite[Theorem 4.1]{joni} again, but we have the technical issue that $\Lambda_{\HB,w}$ does not quite have a good uniform bound, but is instead only controlled in $\ell^k$ norm for arbitrarily large but finite $k$.  However, from Lemma \ref{bounds} (applied with sufficiently large $k$) and Chebyshev's inequality, for any small $\kappa > 0$ and $\eps>0$ we can find an approximation $\Lambda'_{\HB,w}$ to $\Lambda_{\HB,w}$ with
\begin{align}\label{Lambda'2} \| \Lambda_{\HB,w} - \Lambda'_{\HB,w}\|_{\ell^1[N]} \leq \kappa\,\,\textnormal{  and  }\,\, \Lambda'_{\HB,w}(n)=O_\eps( \kappa^{-\eps} \langle \Log w\rangle^{O_\eps(1)}).
\end{align}
We can use the $\ell^1$ norm to control the $u^{d+1}$ norm, hence by~\eqref{eq:CramerHB} and the triangle inequality
\begin{align}\label{Lambda'}
\| \Lambda_{\Cramer,w} - \Lambda'_{\HB,w} \|_{u^{d+1}[N]} \lesssim \kappa + w^{-c}.
\end{align}
Now we can apply \cite[Theorem 4.1]{joni} (and Lemma \ref{cramer-uniform}) to conclude that
$$|\langle \tilde \Avg_{N,\Lambda_{\Cramer,w} - \Lambda'_{\HB,w}}(f,g), h \rangle| \lesssim_\eps  \langle \Log w\rangle^{O_{\varepsilon}(1)}(\kappa^c + \kappa^{-\eps} w^{-c}) N^d.$$
Finally, from the triangle inequality and Cauchy--Schwarz, we can crudely bound
$$|\langle \tilde \Avg_{N,\Lambda'_{\HB,w} - \Lambda_{\HB,w}}(f,g), h \rangle| \lesssim \kappa N^d.$$
Putting this all together, choosing $\eps$ to be sufficiently small, and $\kappa$ to be a small multiple of $w^{-c}$ for a suitable $c$, we conclude that
$$|\langle \tilde \Avg_{N,\Lambda_N - \Lambda_{\HB,w}}(f,g), h \rangle| \lesssim w^{-c} N^d$$
for any $1 \leq w \leq \exp(\Log^{1/C_0} N)$.  In particular, from \eqref{hyp-wt} we now have
$$|\langle \tilde \Avg_{N,\Lambda_{\HB,w}}(f,g), h \rangle| \gtrsim \delta N^d$$
for some $1 \leq w \lesssim \delta^{-O(1)}$.  Expanding \eqref{hb-def} and using the triangle inequality and crude bounds, we conclude that
$$|\langle \tilde \Avg_{N,e(-r\cdot/q)}(f,g), h \rangle| \gtrsim \delta^{O(1)} N^d$$
for some $1 \leq r \leq q \lesssim \delta^{-O(1)}$.  But observe the identity
$$ \langle \tilde \Avg_{N,e(-r\cdot/q)}(f,g), h \rangle = \langle \tilde \Avg_{N,1}(e(-r\cdot/q)f,g), e(-r\cdot/q) h \rangle.$$
We can thus apply Proposition \ref{lip} to conclude that
$$ |\langle e(-r\cdot/q) f, F \rangle| \gtrsim \delta^{O(1)} N^d$$
for some function $F$ obeying the conclusions of that proposition.  Transferring the plane wave $e(-r\cdot/ q)$ from $f$ to $F$, we obtain the claim (noting that the denominator $b$ will remain acceptably under control since $q \lesssim \delta^{-O(1)}$).
\end{proof}

If we now repeat the arguments of \cite[\S 6.1]{KMT0}, using Proposition \ref{lip-wt} and Lemma \ref{lp-improv} in place of \cite[Proposition 6.6]{KMT0} and \cite[Proposition 6.21]{KMT0} respectively, inserting the weights $\Lambda_N$ in the averaging operators in the obvious fashion, we obtain case (i) of Theorem \ref{ssmae}.  To handle case (ii), we need the following variant of Proposition \ref{lip-wt}.

\begin{proposition}[Weighted inverse theorem for $g$]\label{lip-wt-g}  Under the hypotheses of Proposition~\ref{lip-wt}, there exists a function $G \in \ell^2(\Z)$ with
    \begin{equation}\label{ooh}
    \|G\|_{\ell^\infty(\Z)} \lesssim 1; \quad \|G\|_{\ell^1(\Z)} \lesssim N^d
    \end{equation}
    and with $\F_{\Z}G$  supported in the $O(\delta^{-O(1)}/N^d)$-neighborhood of some rational $a/b \mod 1 \in \Q/\Z$ with $b=O(\delta^{-O(1)})$ such that
    \begin{equation}\label{gG}
    |\langle g, G \rangle| \gtrsim \delta^{O(1)} N^d.
    \end{equation}
\end{proposition}

But this can be derived from \cite[Proposition 6.26]{KMT0} in precisely the same way Proposition~\ref{lip-wt} was derived from \cite[Proposition 6.6]{KMT0}. By repeating the remaining arguments of \cite[\S 6.2]{KMT0}, one obtains case (ii) of Theorem \ref{ssmae}.

\section{Remarks}\label{Remarks-sec}

\subsection{Peluse's inverse theorem for the primes}

As is clear from the previous sections, Peluse's inverse theorem~\cite{Peluse} was an important ingredient in the proof of the unweighted bilinear ergodic theorem in~\cite{KMT0}. In the course of proving Theorem~\ref{main-thm}, we essentially needed a version of this inverse theorem where one of the variables was weighted by the approximant $\Lambda_N$; see Proposition~\ref{lip-wt}. It is natural to ask if one can also obtain a version of Peluse's inverse theorem with the von Mangoldt weight $\Lambda$.  We record here how such a result quickly follows from the arguments used to prove Proposition~\ref{lip-wt}.

\begin{theorem}[Peluse's inverse theorem with prime weight]\label{peluse-thm}
Let $k,d\in \mathbb{N}$ and $A>0$. Let $N\geq 2$, $(\log N)^{-A}\leq \delta\leq 1$ and $N_0\sim N^d$. Let $P_1,\ldots, P_k$ be polynomials with integer coefficients of distinct degrees, with maximal degree $d$. Let $h,f_1,\ldots, f_k\colon \Z\to \mathbb{C}$ be functions bounded in modulus by $1$ and supported on $[-N_0,N_0]$. Suppose that
\begin{align}\label{inverse}
\left|\sum_{x\in \Z}\E_{n\in [N]}\Lambda(n)h(x)f_1(x+P_1(n))\cdots f_k(x+P_k(n))\right| \geq \delta N^d.   
\end{align}
Then either $N_0\lesssim_{P_1,\ldots, P_k} \delta^{-O_d(1)}$ or there exists a positive integer $q\lesssim_{P_1,\ldots, P_k} \delta^{-O_d(1)}$ and  $\delta^{O_d(1)}N\lesssim_{P_1,\ldots, P_k} N'\leq N$ such that
\begin{align*}
\frac{1}{N^d}\left|\sum_{x\in \Z}\E_{m\in [N']}f_1(x+qm)\right|\gtrsim_{A,P_1,\ldots, P_k}\delta^{O_d(1)}.   
\end{align*}
\end{theorem}
\begin{proof}
Fix $P_1,\ldots, P_k$; we allow all implied constants to depend on them. Define the polynomial averaging operator 
\begin{align*}
T_{N,\theta}(h,f_1,\ldots, f_k)\coloneqq \sum_{x\in \Z}\E_{n\in [N]}\theta(n)h(x)f_1(x+P_1(n))\cdots f_k(x+P_k(n)).    
\end{align*}
Let $w_0=\delta^{-C_d}$ for a large enough constant $C_d$. 
We claim that 
\begin{align}\label{T1}
T_{N,\Lambda-\Lambda_N}(h,f_1,\ldots, f_k)\lesssim_A (\log N)^{-A} 
\end{align}
and 
\begin{align}\label{T2}
T_{N,\Lambda_N-\Lambda_{\Cramer,w_0}}(h,f_1,\ldots, f_k)\lesssim \delta^2
\end{align}
and 
\begin{align}\label{T3}
T_{N,\Lambda_{\Cramer,w_0}-\Lambda_{\HB,w_0}}(h,f_1,\ldots, f_k)\lesssim \delta^2.     
\end{align}
After we have these three estimates, we conclude from~\eqref{inverse} and linearity that
\begin{align*}
|T_{N,\Lambda_{\HB,w_0}}(h,f_1,\ldots, f_k)|\gtrsim \delta.     
\end{align*}
By~\eqref{hb-def} and~\eqref{cq-def}, the function $\Lambda_{\HB,w_0}$ is a linear combination, with $1$-bounded coefficients, of $O(w_0^3)$ indicators of arithmetic progressions of common difference at most $w_0$. Hence, crudely using the triangle inequality, we obtain
\begin{align*}
|T_{N,\ind{a\ (q')}}(h,f_1,\ldots, f_k)|\gtrsim \delta^{O_d(1)}     
\end{align*}
for some $1\leq a\leq q'\lesssim \delta^{-O_d(1)}$. 
But now the claim of the theorem follows from~\cite[Theorem 3.3]{Peluse} after making a change of variables. 

We are left with showing~\eqref{T1},~\eqref{T2},~\eqref{T3}. The estimate~\eqref{T1} follows immediately from~\cite[Theorem 4.1]{joni} and~\eqref{lambda-approx}. The estimate~\eqref{T2} follows by using Lemma~\ref{cramer-stable}, Lemma~\ref{cramer-uniform} and~\cite[Theorem 4.1]{joni} to obtain
\begin{align*}
T_{N,\Lambda_{\Cramer, w}-\Lambda_{\Cramer,z}}(h,f_1,\ldots, f_k)\lesssim w^{-c_d}
\end{align*}
for some $c_d>0$ and any $z\in [w/2,w]$, $1\leq w\leq \exp((\log N)^{1/10})$, and then summing this dyadically. For proving~\eqref{T3}, note that from~\eqref{Lambda'} and~\cite[Theorem 4.1]{joni}, we have for any  $\kappa>0, \varepsilon>0$ the bound 
\begin{align*}
T_{N,\Lambda_{\Cramer, w_0}-\Lambda'_{\HB,w_0}}(h,f_1,\ldots, f_k)  \lesssim_\eps  \langle \Log w_0\rangle^{O_{\varepsilon}(1)}(\kappa^{c_d'} + \kappa^{-\eps} w_0^{-c_d'}) N^d, 
\end{align*}
with $\Lambda'_{\HB,w_0}$ obeying~\eqref{Lambda'2}. But from~\eqref{Lambda'2} and the triangle inequality we now obtain~\eqref{T3} by taking $\varepsilon>0$ small enough and $\kappa=w_0^{-c}$ for a small enough constant $c$ (depending on $d$). This was enough to complete the proof.
\end{proof}

\subsection{Siegel zeroes}

In this subsection, we mention an alternative approach to Theorem~\ref{main-thm} based on working with Siegel zeroes. This approach is somewhat more complicated than the one implemented above, and we shall only sketch it very briefly, leaving the details to the interested reader. 

The place in the proof of Theorem~\ref{main-thm} where passing from the von Mangoldt function $\Lambda$ to the approximant $\Lambda_N$ avoided dealing with Siegel zeroes is Proposition~\ref{mod-approx}, so we begin by sketching how a variant of Proposition~\ref{mod-approx} can be proven for the weight $\Lambda$.   

We say that a modulus $q\geq 2$ is exceptional if there exists a non-principal real Dirichlet character $\chi_q\pmod q$ such that $L(s,\chi_q)$ has a real zero $\beta_q>1-c_0/(\log q)$, where $c_0$ is some small absolute constant. We call the corresponding character $\chi_q$ an exceptional character, and we call $\beta_q$ a Siegel zero. For any given exceptional $q$, the character $\chi_q$ and Siegel zero $\beta_q$ are uniquely determined.

For exceptional characters $\chi_q$, we define the arithmetic symbol 
\[ m_{\hat{\mathbb{Z}}^{\times},\chi_q}\big( \frac{a_1}{q} \mod 1, \frac{a_2}{q} \mod 1) \coloneqq \
\frac{1}{\phi(q)} \sum_{r\in (\mathbb{Z}/q\mathbb{Z})^{\times}} e\big(\frac{a_1r}{q}  + \frac{a_2 P(r)}{q}\big) \chi_q(r)
,\]
and the (weighted) continuous multiplier
\[ \tilde{m}_{N,\mathbb{R},\chi_q}(\xi_1,\xi_2) \coloneqq \int_{1/2}^{1} e( \xi_1 t + \xi_2 P(t)) t^{\beta_q - 1} \ dt, \]
where $\beta_q\in (0,1)$ is the Siegel zero.  Then if we replace in~\eqref{norma}
\begin{align*}
&\B^{l_1, l_2, m_{\hat \Z^\times}}_{(\eta_{\leq -\Log N+s} \otimes \eta_{\leq -d\Log N+ds})\tilde m_{N,\R} } 
\longrightarrow \\
& \qquad 
\B^{l_1, l_2, m_{\hat \Z^\times}}_{(\eta_{\leq -\Log N+s} \otimes \eta_{\leq -d\Log N+ds})\tilde m_{N,\R} } + \sum_{q \text{ exceptional}} \B^{l_1, l_2, m_{\hat \Z^\times,\chi_q}}_{(\eta_{\leq -\Log N+s} \otimes \eta_{\leq -d\Log N+ds})\tilde m_{N,\R,\chi_q} }, 
    \end{align*}
the conclusion of Proposition \ref{mod-approx} holds with the von Mangdolt weight $\Lambda$ in place of $\Lambda_N$. This follows from essentially the same proof as in Section~\ref{verifying-sec}, but using the Landau--Page theorem (\cite[Corollary 11.10]{mv}) in place of Corollary~\ref{mean}.

In the large-scale regime, the error bounds arising from the Siegel--Walfisz theorem remove the need for the above approximation; in the small-scale regime
\[\{ N \in \mathbb{D}\colon 2^{u^{O(1/(C_0 \rho))}} \leq N \leq 3^{C_0 \cdot 2^u} \}\]
further analysis is required to reduce matters to the two-parameter Rademacher--Menshov inequality.

The first observation is the classical fact that there is at most one exceptional character at each dyadic scale:
\begin{align}\label{e:excj} |\{q\in (2^j,2^{j+1}]\colon q \text{ exceptional} \} \leq 1.
\end{align}
We let $q_j$ denote the unique exceptional modulus in $(2^j,2^{j+1}]$, and abbreviate $\beta_j = \beta_{q_j}$.

 We then introduce a dyadic decomposition 
 \[ \sum_{q \text{ exceptional}} \B^{l_1, l_2, m_{\hat \Z^\times,\chi_q}}_{(\eta_{\leq -\Log N+s} \otimes \eta_{\leq -d\Log N+ds})\tilde m_{N,\R,\chi_q}} = \sum_{j \leq 2^{\rho l}} C_{N,j}(f,g),\]
where
\begin{align*}
    &C_{N,j}(f,g)(x) \\
& = \int_{1/2}^1 \Big( \int_{\mathbb{T}^2}  \sum_{(a_1/q_j,a_2/q_j)\colon \Height(a_i/q_j) = 2^{l_i}}
m_{\hat{\mathbb{Z}}^{\times},\chi_{q_j}}(a_1/q_j,a_2/q_j) e(a_1 x/q_j + a_2 x/q_j)  \\
& \qquad\times 
\big( \hat{f}(\xi_1 + a_1/q_j) \cdot \varphi(2^u \xi_1) \cdot e(\xi_1 N t) \big) \\
& \qquad \qquad\times \big( \hat{g}(\xi_2 + a_2/q_j) \cdot \varphi(2^{du} \xi_2) \cdot e( \xi_2 P(Nt)) \big) e(\xi_1 x + \xi_2 x) \cdot N^{\beta_{j} - 1} t^{\beta_{j} - 1} \ d\xi_1 d\xi_2 \Big)   \ dt.
\end{align*}

The key novelty then derives from proving the following modified Rademacher--Menshov-type inequality, similar to \cite[Lemma 8.2]{KMT0}.

\begin{lemma}\label{l:varbound}
    Let $V,W$ be normed vector spaces, $K,J$ be two positive integers, and  let $0<q<\infty$. Let  $B_j\colon V\times W\rightarrow L^q(X)$ be a family of bilinear operators for $j\in [J]$. Let  $\{f_{k}^j\}, \{g_k^j\}$ be sets of functions with $f_k^j\in V$ and $g_k^j\in W$ for $k\in [K]$ and $j\in [J]$. Then
\begin{align*}&\Big\|V^2\big(\sum_{j\in [J]}B_j(f_k^j,g_k^j)\colon k \in [K]\big)\Big\|_{L^q(X)} \\
    &\qquad\lesssim_{q} \langle \Log K\rangle^{O_q(1)} \sup_{\epsilon_{k}^{j}, {\varepsilon}_k^j \in \{ \pm 1 \}} \Big\| \sum_{j\in [J]}B_j\big(\sum_{k\in [K]} \epsilon_k^j  (f_k^j- f_{k-1}^j), \sum_{k\in [K]} {\varepsilon}_k^j  (g_k^j- g_{k-1}^j)\big)\Big\|_{L^q(X)}.
\end{align*}
\end{lemma}
This result may be of independent interest, so we provide a brief proof.
\begin{proof}
    Set $a_{k_1,k_2}= \sum_{j \in [J]} B_j(f_{k_1}^j,g_{k_2}^j)$. By \cite[Lemma 8.1]{KMT0}, we have
    \begin{align*}
        V^2\big(\sum_{j \in [J]} B_j(f_k^j,g_k^j)\colon k\in [K]\big) \lesssim \sum_{\substack{M_1,M_2<K\\M_1,M_2\colon \textup{dyadic}}} \parallel \Delta \sum_{j \leq J} B_j(f_{M_1n_1}^j,g_{M_2n_2}^j)\parallel_{\ell^2(n_1,n_2)}
    \end{align*}
    where 
    \begin{align*}
        \Delta \sum_{j \in [J]} B_j(f_{M_1n_1}^j,g_{M_2n_2}^j) =& \sum_{ j \in [J]} B_j(f_{M_1n_1}^j,g_{M_2n_2}^j) - \sum_{j\in [J]} B_j(f_{(n_1-1)M_1}^j,g_{M_2n_2}^j)\\
        &- \sum_{j \in [J]} B_j(f_{M_1n_1}^j,g_{(n_2-1)M_2}^j) + \sum_{j \in [J]} B_j(f_{(n_1-1)M_1}^j,g_{(n_2-1)M_2}^j)
    \end{align*}
    Taking 
    \[ \tilde f_{M_1n_1} = f_{M_1n_1} - f_{(n_1-1)M_1}, \; \; \;  \tilde g_{M_2n_2} = g_{M_2n_2} - g_{(n_2-1)M_2},\]
    we need to bound
    \begin{align}
        \langle \Log K\rangle^{O_q(1)} \sup_{M_1,M_2 < K \text{ dyadic}} \Big\| \Big(\sum_{\substack{n_1<k/M_1\\n_2<k/M_2}}| \sum_{j\in [J]} B_j(\tilde f_{M_1n_1}^j,\tilde g_{M_2n_2}^j)|^2\Big)^{1/2}\Big\|_{L^q(X)}
    \end{align}
    Applying Khintchine's inequality
    \[ \big( \sum_{n} |a_n|^2 \big)^{1/2} = \big( \mathbb{E}_{\epsilon_n \in \pm 1}| \sum_n \epsilon_n a_n|^2 \big)^{1/2} \sim_q \big( \mathbb{E}_{\epsilon_n \in \pm 1}| \sum_n \epsilon_n a_n|^q \big)^{1/q}, \]
    we arrive at the following chain of inequalities:
    \begin{align*}
        &\Big\| V^2\big(\sum_{j\in [J]}B_j(f_k^j,g_k^j)\colon k \in [K]\big)\Big\|_{L^q(X)} \\
        &\qquad\lesssim \langle \Log K\rangle^{O_q(1)} \sup_{M_1,M_2} \parallel \big(\mathbb{E}_{\varepsilon_{n_2} \in \pm 1} \sum_{\substack{n_1}}| \sum_{n_2}\sum_{j\in [J]} \varepsilon_{n_2} B_s(\tilde f_{M_1n_1}^j,\tilde g_{M_2n_2}^j)|^2\big)^{1/2}\parallel_{L^q(X)} \\
        & \qquad \qquad\lesssim \langle \Log K\rangle^{O_q(1)} \sup_{M_1,M_2} \parallel \big(\mathbb{E}_{\epsilon_{n_1}, \varepsilon_{n_2} \in \pm 1}| \sum_{\substack{n_1}} \sum_{n_2}\sum_{j\in [J]} \epsilon_{n_1} \varepsilon_{n_2} B_j(\tilde f_{M_1n_1}^j,\tilde g_{M_2n_2}^j)|^2\big)^{1/2}\parallel_{L^q(X)} \\
        & \qquad \qquad \qquad \lesssim_q \langle \Log K\rangle^{O_q(1)} \sup_{M_1,M_2} \parallel \big(\mathbb{E}_{\epsilon_{n_1},  \varepsilon_{n_2} \in \pm 1} |\sum_{\substack{n_1}} \sum_{n_2}\sum_{j\in [J]} \epsilon_{n_1} \varepsilon_{n_2} B_j(\tilde f_{M_1n_1}^j,\tilde g_{M_2n_2}^j)|^q\big)^{1/q}\parallel_{L^q(X)} \\
        & \qquad \qquad \qquad \qquad \lesssim_q \langle \Log K\rangle^{O_q(1)} \sup_{M_1,M_2, \epsilon_{n_1},\epsilon_{n_2}} \parallel  \sum_{\substack{n_1}} \sum_{n_2}\sum_{j\in [J]} \epsilon_{n_1} \varepsilon_{n_2} B_j(\tilde f_{M_1n_1}^j,\tilde g_{M_2n_2}^j)\parallel_{L^q(X)}.
    \end{align*}
    By bilinearity, we may consolidate
    \[ \sum_{\substack{n_1}} \sum_{n_2}\sum_{j\in [J]} \epsilon_{n_1} \varepsilon_{n_2} B_j(\tilde f_{M_1n_1}^j,\tilde g_{M_2n_2}^j) = \sum_{j\in [J]} B_j(\sum_{n_1} \epsilon_{n_1} \tilde f_{M_1 n_1}^j, \sum_{n_2} \varepsilon_{n_2}\tilde g_{M_2n_2}^j);\]
    putting everything together
    \begin{align*}
        &\parallel {V}^2\big(\sum_{j\in [J] }B_j(f_k^j,g_k^j)\colon k \in [K]\big)\parallel_{L^q(X)} \\
        &\lesssim \langle \Log K\rangle^{O_q(1)} \sup_{\substack{M_1,M_2\\ \epsilon_{n_1}, \varepsilon_{n_2}}} \Big\| \sum_{j \in [J]} B_j\left( \sum_{n_1} \epsilon_{n_1} ( f_{M_1n_1}^j- f_{(n_1-1)M_1}^j),\sum_{n_2}\varepsilon_{n_2}  (g_{M_2n_2}^j-g_{(n_2-1)M_2}^j)\right)\Big\|_{L^q(X)},
    \end{align*}
    and so we get the result upon telescoping e.g.
    \[ \epsilon_{n_1} (f_{M_1n_1}^j- f_{(n_1-1)M_1}^j) =  \sum_{(n_1 - 1)M_1 < k \leq M_1 n_1} \epsilon_{n_1}( f_k^j - f_{k_1}^j ) \eqqcolon \sum_{(n_1 - 1)M_1 < k \leq M_1 n_1} \epsilon_k^j (f_k^j - f_{k-1}^j).\]
    \end{proof}

\subsection{Breaking Duality}\label{sec:duality}
We briefly remark that one may establish Theorem \ref{main-thm} with $r$-variation restricted to the range $r > 2 + \epsilon$ for exponents $p_1,p_2 >1$ that satisfy 
\[ 1 <  \frac{1}{p} \coloneqq \frac{1}{p_1} + \frac{1}{p_2} < 1 + \epsilon', \]
where $\epsilon' > 0$ is sufficiently small in terms of $\epsilon$, hence going beyond the duality range.

The single-scale estimate
\begin{align}\label{e:smallp} \| A_{N;\Lambda;X}(f,g) \|_{L^p(X)} \lesssim \| f \|_{L^{p_1}(X)}  \| g \|_{L^{p_2}(X)},
\end{align}
anchors the argument; \eqref{e:smallp} follows from H\"{o}lder's inequality and the improving estimate Lemma \ref{lp-improv}, as per \cite[Lemma 11.1]{KMT0}. With \eqref{e:smallp} in hand, the proof of \cite[Proposition 11.4]{KMT0} can be formally reproduced, with only notational changes arising. We leave the details to the interested reader.

\subsection{Sharpness of the variational result}\label{sec:sharpness}
The unboundedness of the quadratic variation along polynomial orbits, namely \cite[Proposition 12.1]{KMT0}, extends to our context.
\begin{proposition}
    Let $P \in \Z[\n]$ be a non-constant polynomial, and let $0 < p \leq \infty$. Let $I \subset \mathbb{N}$ be an infinite set. Then for every $C > 0$ there exists a measure-preserving system $(X,\mu,T)$ of total measure 1 and a $1$-bounded $f \in L^{\infty}(X)$ so that
    \[ \| \big( \mathbb{E}_{p \in [N]} T^{P(p)} f \big)_{N \in I} \|_{L^p(X;V^2)} \geq C. \]
        \end{proposition}

We shall leave the details of the proof of this proposition to the interested reader as it is similar to the proof of \cite[Proposition 12.1]{KMT0}. The key additional observation is the equidistribution of
\[ p \mapsto \big( \alpha_1 \cdot P(p),\dots,\alpha_K \cdot P(p) \big) \subset \mathbb{T}^K \]
over the primes whenever $\alpha_1,\dots,\alpha_K$ are $\mathbb{Q}$-linearly independent, and $P \in \Z[\n]$ is a non-constant polynomial (which follows from Weyl's criterion and a standard exponential sum estimate for polynomials of primes; see e.g.~\cite[Theorem 1.3]{mat1}). 
        
To see why this implies the sharpness of the range of the variational estimate in Theorem~\ref{main-thm}, one may employ the convexity arguments of \cite[\S 5]{MTZK}, taking into account \cite[Proposition 4.1]{MTZK}, to obtain the lower bound
\[ \| \big( \mathbb{E}_{p \in [N]} T^{P(p)} f \big)_{N \in I} \|_{L^p(X;V^2)} \leq \| \big( \mathbb{E}_{n \in [N]} \Lambda(n) \cdot T^{P(n)} f \big)_{N \in I} \|_{L^p(X;V^2)} + O(1). \]

\subsection{Continuous Extensions}
From the perspective of density, the primes are ``full\-/dimensional", with a very ``Fourier-uniform" measure, $\Lambda$. A natural question concerns establishing a continuous analogue of Theorem \ref{main-thm}, namely the existence of a measure $\nu$ supported on $[0,1]$, with (say) full Fourier dimension,
\[ |\hat{\nu}(\xi)| \lesssim (1 + |\xi|)^{o(1)-1/2}\]
so that
\[ \lim_{N \to \infty} \frac{1}{N} \int_0^N f(x-t) g(x-P(t)) \ d\nu(t), \; \; \; d = \text{deg}(P) \geq 2 \]
exists almost everywhere whenever $f \in L^{p_1}(\mathbb{R})$ and $g \in L^{p_2}(\mathbb{R})$ with $p_1,p_2 > 1$ and $\frac{1}{p_1} + \frac{1}{p_2} \leq 1$. The key point is establishing a suitable Sobolev inequality, namely
\[ \Big\| \frac{1}{N} \int_0^N f(x-t) g(x-P(t)) \ d\nu(t) \Big\|_{L^1([0,CN^d])} \lesssim \big(2^{-cl} + O_A(\langle \log N \rangle^{-A} )\big)  N^d\]
for some $c > 0$, whenever $|f|, |g| \leq 1$, and $\hat{f}$ vanishes on $\{ |\xi| \lesssim 2^{l}/N \}$ and/or $\hat{g}$ vanishes on $\{ |\xi| \lesssim 2^l/N^d\}$.

Estimates of this form in the unweighted setting go back to \cite{BRoth}, with the strongest estimates recently established by one of us as part of a much more general phenomenon, see \cite{KMPW}. This approach relies on PET induction, which suggests that certain Gowers-uniformity conditions might need be imposed on $\nu$; it is unclear how this might interact with dimension, so we leave the problem to the interested reader.

\end{document}